\documentclass [11pt, a4paper]{article}
\usepackage[utf8]{inputenc}
\usepackage[T1]{fontenc}
\usepackage{amsmath}
\usepackage{amsfonts}
\usepackage{amsthm}
\usepackage{color}
\usepackage[normalem]{ulem}

\newlength{\oldparindent}
\newcommand{\bpf}{\begin{preuve}}
\newcommand{\epf}{\end{preuve}\medskip}

\newcommand{\benum}{\begin{enumerate}}
\newcommand{\eenum}{\end{enumerate}}

\newcommand{\bitem}{\begin{itemize}}
\newcommand{\eitem}{\end{itemize}}

\newcommand{\brmq}{\begin{rmq}}
\newcommand{\ermq}{\end{rmq}}

\newcommand{\brmqs}{\begin{rmqs}}
\newcommand{\ermqs}{\end{rmqs}}

\newcommand{\bapp}{\begin{application}}
\newcommand{\eapp}{\end{application}}

\newcommand{\bapps}{\begin{applications}}
\newcommand{\eapps}{\end{applications}}

\newcommand{\bdefi}{\begin{definition}}
\newcommand{\edefi}{\end{definition}}

\newcommand{\bques}{\begin{question}}
\newcommand{\eques}{\end{question}}
\newcommand{\bcon}{\begin{conjecture}}
\newcommand{\econ}{\end{conjecture}}
\newcommand{\beq}{\begin{equation}}
\newcommand{\eeq}{\end{equation}}

\def\bpm{\begin{pmatrix}}
\def\epm{\end{pmatrix}}

\newcommand{\bcas}{\begin{cases}}
\newcommand{\ecas}{\end{cases}}

\newcommand{\bex}{\begin{exemp}}
\newcommand{\eex}{\end{exemp}}

\newcommand{\bexs}{\begin{exemps}}
\newcommand{\eexs}{\end{exemps}}

\newcommand{\bprop}{\begin{proposition}}
\newcommand{\eprop}{\end{proposition}}

\newcommand{\bthm}{\begin{theorem}}
\newcommand{\ethm}{\end{theorem}}

\newcommand{\bcor}{\begin{corollary}}
\newcommand{\ecor}{\end{corollary}}

\newcommand{\blem}{\begin{lemma}}
\newcommand{\elem}{\end{lemma}}

\newcommand{\beqna}{\begin{eqnarray}}
\newcommand{\eeqna}{\end{eqnarray}}

\newcommand{\beqnas}{\begin{eqnarray*}}
\newcommand{\eeqnas}{\end{eqnarray*}}

\definecolor{green}{rgb}{0,.7,.2}
\definecolor{orange}{rgb}{0.9,.5,0}

\newcommand{\bbN}{{\mathbb {N}}}
\newcommand{\bbR}{{\mathbb {R}}}

\def\bf\Sigma{{\mathbf{\Sigma}}}

\newtheorem{theorem}{Theorem}[section]
\newtheorem{lemma}[theorem]{Lemma}
\newtheorem{definition}[theorem]{Definition}
\newtheorem{proposition}[theorem]{Proposition}
\newtheorem{corollary}[theorem]{Corollary}

\newtheorem{question}{Question}[section]
\newtheorem{conjecture}{Conjecture}[section]

\newenvironment{exemp}{\noindent{\bf Example. --- }}{\par}
\newenvironment{exemps}{\noindent{\bf Examples}\benum}{\eenum\par}
\newtheorem{remark}[theorem]{Remark}

\newenvironment{application}{\noindent{\bf Application. --- }}{\par}
\newenvironment{applications}{\noindent{\bf Applications. ---
}\benum}{\eenum\par}

\theoremstyle{definition}

\date{\today}

\begin{document}

\title{\textbf{\large Langevin deformation for R\'enyi entropy on Wasserstein space over Riemannian manifolds}}
\author{\ \ Rong Lei\thanks{Academy of Mathematics and Systems Science, Chinese Academy of Sciences, No. 55, Zhongguancun East Road, Beijing 100190, China. Research of R. Lei has been supported by National Key R$\&$D Program of China (No. 2020YF0712700), NSFC No. 12171458, and partially supported by JSPS Grant-in-Aid for Scientific Research (S) No. A22H04942. }, \ \ Songzi Li\thanks{School of Mathematics, Renmin University of China, Beijing 100872, China. Research of S. Li has been supported by NSFC No. 11901569.}, \ \ Xiang-Dong Li
\thanks{Academy of Mathematics and Systems Science, Chinese Academy of Sciences, No. 55, Zhongguancun East Road, Beijing 100190, China, and School of Mathematical Sciences, University of Chinese Academy of Sciences, Beijing 100049, China. Research of X.-D. Li has been supported by National Key R$\&$D Program of China (No. 2020YF0712700), NSFC No. 12171458, and Key Laboratory RCSDS, CAS, No. 2008DP173182.} }
  
\maketitle

\begin{abstract}  We introduce the Langevin deformation for the R\'enyi entropy on the $L^2$-Wasserstein space over $\mathbb{R}^n$ or a Riemannian manifold, which interpolates between the porous medium equation and the Benamou-Brenier geodesic flow on the $L^2$-Wasserstein space and can be regarded as the compressible Euler equations for isentropic gas with damping.
We prove the $W$-entropy-information formulae and the the rigidity theorems for the Langevin deformation for the R\'enyi entropy on the Wasserstein space over complete Riemannian manifolds with non-negative Ricci curvature or  CD$(0, m)$-condition. Moreover, we prove the monotonicity of the Hamiltonian and the convexity of the Lagrangian along the Langevin deformation of flows. Finally, we prove the convergence of the Langevin deformation for the R\'enyi entropy as $c\rightarrow 0$ and $c\rightarrow \infty$ respectively. Our results are new even in the case of Euclidean spaces and compact or complete Riemannian manifolds with non-negative Ricci curvature.  
\end{abstract}


\medskip



\section{Introduction}
\subsection{Background and motivation}
Entropy is essentially related to the second law of thermodynamics and plays a significant role in  statistical mechanics. The famous ${\mathrm H}$-theorem for the Boltzmann kinetic equation indicates that the entropy of the evolution process of ideal gas  increases and the Maxwellian distribution is the local equilibrium distribution for the velocity of the ideal gas. Moreover,  entropy plays an important role in information theory \cite{Shan48}, and the study of conservation law systems and hydrodynamic equations, see Landau \cite{Lan} and Cercignani \cite{Cercignani} etc.
 
In 2002, G. Perelman \cite{Perelman2002} introduced the $W$-entropy for the Ricci flow and proved its monotonicity along the conjugate heat equation, which implies the non-local collapsing theorem for the Ricci flow and ``removes the major stumbling block in Hamilton's approach to geometrization''. This  plays an important 
role in the final resolution of the Poincar\'e conjecture. For more description and explanation of Perelman’s beautiful work, see \cite{BCGG} and references therein. Since then, the monotonicity of the $W$-entropy has been extended to other  geometric flows on Riemannian manifolds, for example, the heat equation (see e.g. \cite{Ni, Li-Xu,Li2012MA,Kuwada-Li})  and non-linear diffusion equations (especially the porous medium equation and fast diffusion equation, see \cite{LNVV}). 

On the other hand, by the works of Jordan-Kinderlehrer-Otto \cite{JKO} and Otto \cite{Otto2001}, entropy is the natural free energy functional in the study of Fokker-Planck equations in non-equilibrium statistical mechanics. In \cite{Otto2001}, Otto introduced a natural infinite dimensional Riemannian metric on the $L^2$-Wasserstein space of probability distributions on $\mathbb{R}^n$, and proved that the heat equation on $\mathbb{R}^n$ 
$$\partial_t u=\Delta u$$
is the gradient flow of the Boltzmann-Shannon entropy 
$$\mathrm{Ent}(u)=\int_{\mathbb{R}^n} u\log u dx,$$
and the porous medium equation ($\gamma>1$) or the fast diffusion equation ($\gamma<1$) on $\mathbb{R}^n$
$$\partial_t u=\Delta u^\gamma, \gamma\neq 1$$
is the gradient flow of the R\'enyi entropy
$$\mathrm{Ent}_\gamma(u)={\frac{1}{\gamma-1}}\int_{\mathbb{R}^n} u^\gamma dx,\quad \gamma\neq1, $$
on the $L^2$-Wasserstein space $\mathcal{P}_2(\mathbb{R}^n,dx)$ with the infinite dimensional Riemannian metric introduced in \cite{Otto2001} respectively. For more detail, see \cite{V1, V2} and Section $2$ below. 

Inspired by the result of Perelman \cite{Perelman2002}, the third author of this paper \cite{Li2012MA}  proved the $W$-entropy formula for the heat equation $\partial_t u=Lu$ associated with the Witten Laplacian $L$ (see $(\ref{Witten Laplacian})$ below for its definition) on complete Riemannian manifolds with CD$(0, m)$-condition (for its definition, see Section 2 below), which extends a previous result due to Ni \cite{Ni} on the $W$-entropy formula for the heat equation $\partial_t u=\Delta u$ on compact Riemannian manifolds with non-negative Ricci curvature.  Moreover, inspired by the works of Benamou-Brenier \cite{BB1999}, Lott-Villani \cite{Lott-Villani}  and Lott \cite{Lott09}, the second and the third authors  of this paper  \cite{Li-Li} introduced the $W$-entropy  and proved its monotonicity and rigidity theorem for the geodesic flow on the $L^2$-Wasserstein space over a complete Riemannian manifold with bounded geometry condition equipped with Otto’s infinite dimensional Riemannian metric. Moreover, Li-Li \cite{Li-Li} observed that there is a strong similarity between the $W$-entropy formula for the heat equation of the Witten Laplacian and  the $W$-entropy formula for  the geodesic flow on the Wasserstein space. 

To better understand this similarity,  Li-Li \cite{Li-Li} introduced the so-called Langevin deformation of flows on the $L^2$-Wasserstein space over a Riemannian manifold. More precisely, let $(M, g)$ be a complete Riemannian manifold,  $f\in C^2(M)$,  and $d\mu=e^{-f}dv$ be a weighted volume measure. Let $\mathcal{V}(\rho)=\int_M V(\rho(x))d\mu(x)$ be a  free energy (also called potential) on the Wasserstein space $\mathcal{P}_2(M, \mu)$ (see Section 2 for its definition), where $V\in C^1(\mathbb{R}^+, \mathbb{R})$.  Let $(\rho, \phi)$ be a smooth solution to the following equations, called the Langevin deformation of flows associated with the potential $\mathcal{V}$ on $\mathcal{P}_2(M, \mu)$, 
 
\begin{equation}\label{Lang-V}
	\begin{cases}
		\partial_t\rho-\nabla_{\mu}^*(\rho\nabla\phi)=0,\\
		c^2\left(\partial_t\phi+\frac{1}{2}|\nabla\phi|^2\right)=-\phi-{\delta\mathcal{V}\over \delta \rho},
	\end{cases}
\end{equation}
where $\nabla_\mu^*$ is the $L^2$-adjoint of the Riemannian gradient operator $\nabla$  with respect to the weighted volume measure $\mu$ on $(M, g)$, $c\in [0, \infty)$ is a parameter, and ${\delta \mathcal{V}\over \delta \rho}$ is the $L^2$-derivative of $\mathcal{V}$ with respect to $\mu$. Indeed, we have
$${\delta \mathcal{V}\over \delta \rho}(x)=V'(\rho(x)), \ \ \ \forall x\in M.$$

In particular, if the free energy $\mathcal{V}$ is given by the Boltzmann-Shannon entropy, i.e., 
$$\mathcal{V}(\rho)=\mathrm{Ent}(\rho)=\int_M \rho \log \rho d\mu,$$
then 
$${\delta \mathcal{V}\over \delta \rho}(x)=\log\rho(x)+1,$$ and  the Langevin deformation of flows $(\ref{Lang-V})$ becomes 
\begin{equation}\label{Lang-BS}
	\begin{cases}
		\partial_t\rho-\nabla_{\mu}^*(\rho\nabla\phi)=0,\\
		c^2\left(\partial_t\phi+\frac{1}{2}|\nabla\phi|^2\right)=-\phi-\log \rho-1.
	\end{cases}
\end{equation}

As pointed out in \cite{Li-Li}, at least heuristically, the limiting cases $c=0$ and $c=\infty$ can be specified as follows:
\begin{itemize}
\item when $c=0$, the second equation reduces to $\phi=-\log \rho-1$, and it turns out that the first equation in $(\ref{Lang-BS})$
becomes the gradient flow of the Boltzmann-Shannon entropy, i.e., the heat equation associated with the Witten Laplacian $L$ on Riemaniann manifold
\begin{equation*}\label{HeatE-L}	
		\partial_t\rho=L\rho.
\end{equation*}
\item when $c=\infty$, the second equation in $(\ref{Lang-BS})$
 reduces to the Hamilton-Jacobi equation $\partial_t\phi+{1\over 2}|\nabla \phi|^2=0$, and the whole system becomes the geodesic flow which was first introduced by Benamou-Brenier \cite{BB1999} on the $L^2$-Wasserstein space over $\mathbb{R}^n$ and later by Lott-Villani \cite{Lott-Villani} over Riemannian manifold $M$ (see also Li-Li \cite{Li-Li}), namely, 
 
\begin{equation}\label{Geodesic-BB}
	\begin{cases}
		\partial_t\rho-\nabla_{\mu}^*(\rho\nabla\phi)=0,\\
		\partial_t\phi+\frac{1}{2}|\nabla\phi|^2=0.
	\end{cases}
\end{equation}
\end{itemize}

In view of this, the Langevin deformation of flows $\eqref{Lang-BS}$  is a natural interpolation between the gradient flow of the Boltzmann-Shannon entropy and the Benamou-Brenier geodesic flow on the Wasserstein space over a Riemannian manifold. 

Moreover,  Li-Li \cite{Li-Li} also pointed out that the Langevin deformation of flows \eqref{Lang-BS} can be regarded as the potential flow of the compressible Euler equation with damping on Riemannian manifolds. In the special case where $(M,\mu)=(\mathbb R^n, dx)$ and $\mathcal{V}$ is given by the typical internal energy $\mathrm{Ent}$ or $\mathrm{Ent}_\gamma$, letting $u=\nabla\phi$ and taking $c=1$ in $(\ref{Lang-V})$, then the Langevin deformation of flows becomes the following isothermal Euler equations with damping when $\gamma=1$ or the isentropic (also called polytropic) Euler equations with damping when  $\gamma>1$, i.e., 
\begin{equation*}\label{Euler equation}
\begin{cases}
\partial_t\rho+\nabla\cdot(\rho u)=0,\\
\partial_t u+u\cdot\nabla u=-u-\frac{\nabla P}{\rho},	
\end{cases}
\end{equation*}
where $P$ is the corresponding thermodynamical pressure given by 
\begin{equation*}
P(\rho)=\rho^\gamma, \quad \gamma\geq 1.
\end{equation*}
In the thermodynamics and statistical mechanics, it is well-known that the internal energy functional $\mathrm{Ent}$ or $\mathrm{Ent}_\gamma$ is the Helmholtz free energy and the Gibbs free energy of the isothermal ($\gamma=1$) or isentropic ($\gamma>1$) process up to some constants. See e.g. \cite{Cercignani, LQ}.

The local and global existence, uniqueness and regularity of the Langevin deformation $\eqref{Lang-V}$ on the Wasserstein space over the Euclidean space and a compact Riemannian manifold have been established in \cite{Li-Li} for $\mathcal{V}=\mathrm{Ent}$ and $\mathcal{V}=\mathrm{Ent}_\gamma$. In particular, when $\mathcal{V}=\mathrm{Ent}$, the $W$-entropy formula and the rigidity theorem were proved for the Langevin deformation of flows $(\ref{Lang-BS})$ in \cite{Li-Li}. In this case, the $W$-entropy formula has been modified by incorporating an additional term of information and it is called the $W$-entropy-information formula in \cite{Li-Li}. The rigidity theorem was also proved for the Langevin deformation $(\ref{Lang-BS})$ on the $L^2$-Wasserstein space over complete Riemannian manifolds with CD$(0, m)$-condition, and the convergence of the solution to $(\ref{Lang-BS})$ were also proved for $c\rightarrow 0$ and $c\rightarrow \infty$ respectively in \cite{Li-Li}.

It is natural to ask whether the above results hold for the case where the free energy functional in $(\ref{Lang-V})$ is the R\'enyi entropy. i.e.,  
$$\mathcal{V}(\rho)=\mathrm{Ent}_\gamma(\rho)=\int_M{\rho^\gamma\over\gamma-1}d\mu,  \quad \gamma> 1. $$
This corresponds to the ideal isotropy gases model in thermodynamics, and the Langevin deformation for the R\'enyi entropy on the $L^2$-Wasserstein space over $\mathbb{R}^n$ or a Riemannian manifold becomes 
\begin{equation}\label{Lang-gamma}
	\begin{cases}
		\partial_t\rho-\nabla_{\mu}^*(\rho\nabla\phi)=0,\\
		c^2\left(\partial_t\phi+\frac{1}{2}|\nabla\phi|^2\right)=-\phi-\frac{\gamma\rho^{\gamma-1}}{\gamma-1}.
	\end{cases}
\end{equation}
Extending the argument in \cite{Li-Li}, we prove the $W$-entropy-information formula for the Langevin deformation $(\ref{Lang-gamma})$ and the convergence of the Langevin deformation $(\ref{Lang-gamma})$ as $c\rightarrow 0$ and $c\rightarrow \infty$ respectively. Moreover, we prove the rigidity theorem for the Langevin deformation $(\ref{Lang-gamma})$ on the $L^2$-Wasserstein space over complete Riemannian manifolds with CD$(0, m)$-condition. On the other hand, corresponding to the R\'enyi entropy, we can define a new $W$-entropy for the geodesic flow on the Wasserstein space and prove its monotonicity and rigidity theorem. This extends the results in Li-Li \cite{Li-Li}. Our results are new even in the case of Euclidean space and complete Riemannian manifolds with non-negative Ricci curvature.

The rest of this paper is organized as follows. In Section $1.2$, we introduce some notations. In Section $1.3$, we state the main results of this paper. In Section $2$, we prove the variational formulae for the R\'enyi entropy. In Section $3$, we introduce the rigidity model and prove the $W$-entropy formulae, $W$-entropy-information formulae and rigidity theorems. In Section $4$, we prove the monotonicity of the Hamiltonian and  the convexity of the Lagrangian along the Langevin deformation of flows. In Section $5$, we prove the convergence of Langevin deformation for $c\rightarrow 0$ and $c\rightarrow \infty$ respectively.

\subsection{Notations}

We now introduce some notations and definitions.

Let $(M, g)$ be a complete Riemannian manifold equipped with a weighted volume measure $d \mu=e^{-f} d v$, where $f \in C^2(M)$, and $dv(x)=\sqrt{{\mathrm{det}}(g(x))}dx$ denotes the volume measure on $(M, g)$. Let $\nabla_\mu^*$ be the $L^2$-adjoint of the Riemannian gradient $\nabla$ with respect to the weighted volume measure $d\mu$  on $(M, g)$. More precisely, for any smooth vector field $X$ on $M$, 
$$\nabla_\mu^* X=e^{f}\nabla^*(e^{-f}X),$$ where $\nabla^*=-\nabla\cdot$ is the $L^2$-adjoint of the Riemannian gradient $\nabla$ with respect to the standard volume measure $dv$  on $(M, g)$. The Witten Laplacian on $(M, g, \mu)$ is defined by 
\begin{eqnarray}\label{Witten Laplacian}
L=-\nabla_\mu^*\nabla=\Delta-\nabla f \cdot \nabla,
\end{eqnarray}
which is  a self-adjoint and negatively definite operator on $L^2(M, \mu)$. By It\^{o}’s theory, $L$ is the infinitesimal generator of a diffusion process $X_t$ on $M$ which solves the following Stratonovich SDE
\begin{equation*}
dX_t=U_t\circ dW_t-\nabla f(X_t)dt,
\end{equation*}
where $U_t : T_{X_0}M\to T_{X_t}M$ is the stochastic parallel transport with respect to the Levi-Civita connection along the path $(X_s,s\in[0,t])$, and $W_t$ is a standard Brownian motion on $T_{X_0}M$.

Recall the Bochner-Weitzenböck formula for the Witten Laplacian \cite{BE} 
\begin{equation*}
L|\nabla u|^2-2\langle\nabla u,\nabla Lu\rangle=2\|\nabla^2 u\|^2_{\mathrm{HS}}+2\mathrm{Ric}(L)(\nabla u,\nabla u), \quad \forall u\in C^\infty(M),
\end{equation*}
where $\|\nabla^2 u\|_{\mathrm{HS}}$ denotes the Hilbert–Schmidt norm of $\nabla^2u$ (i.e., the Hessian of u), and
\begin{equation*}
	\mathrm{Ric}(L)=\mathrm{Ric}+\nabla^2 f
\end{equation*}
is the so-called infinite dimensional Bakry–Emery Ricci curvature associated with the Witten Laplacian $L$ on $(M, g, \mu)$. 
%

Let $m>n$ be a constant. The $m$-dimensional Bakry-Emery Ricci curvature associated with the Witten Laplacian $L$ is defined by 
\begin{equation*}
	\operatorname{Ric}_{m, n}(L):=\operatorname{Ric}+\nabla^2 f-\frac{\nabla f \otimes \nabla f}{m-n} .
\end{equation*}
We make the convention that if $m=n$, then $\operatorname{Ric}_{n, n}(L)$ is defined only when $f$ is identically a constant and $\operatorname{Ric}_{n, n}(L)=\rm{Ric}$. By \cite{Lott03,Li-2005}, when $m>n$ is an integer, $\operatorname{Ric}_{m, n}(L)$ is the horizontal projection of the Ricci curvature on the product manifold $\widetilde{M}=M \times N$ equipped with the warped product metric $\widetilde{g}=g \otimes e^{-\frac{2 f}{m-n}} g_N$, where $\left(N, g_N\right)$ is an $(m-n)$-dimensional complete Riemannian manifold. 
Following \cite{BE, Li-Li}, we say that a weighted Riemannian manifold $(M, g, \mu)$ satisfies the CD$(K, m)$-condition for some constants $K \in \mathbb{R}$ and $m \in[n, \infty]$ if and only if 
\begin{equation*}
	\operatorname{Ric}_{m, n}(L) \geq K.
\end{equation*}


The space of all probability measures on $M$ with density function $\rho$ with respect to the reference measure $\mu$ and with finite second moment is denoted by $\mathcal{P}_2(M, \mu)$, i.e., 
\begin{eqnarray*}
\mathcal{P}_2(M, \mu)=\left\{\rho d\mu: \rho\geq 0, \ \int_M \rho d\mu=1\text{ and }  \int_M d(o, x)^2 \rho(x)d\mu(x)<+\infty\right\},
\end{eqnarray*}
where $o$ is some (and hence all) reference point in $M$. When we restrict to smooth density functions, we introduce  $\mathcal{P}_2^\infty(M, \mu)$ as follows   
\begin{eqnarray*}
\mathcal{P}_2^\infty(M, \mu)=\left\{\rho d\mu: \rho\in C^\infty(M, \mathbb{R}^+), \ \int_M \rho d\mu=1 \text{ and }  \int_M d(o, x)^2 \rho(x)d\mu(x)<+\infty\right\}. 
\end{eqnarray*}
Throughout this paper, we call $\mathcal{P}_2(M, \mu)$ the $L^2$-Wasserstein space over $(M, g, \mu)$ or briefly the Wasserstein space, and $\mathcal{P}_2^\infty(M, \mu)$ the $L^2$-smooth Wasserstein space.  

According to Otto \cite{Otto2001} and Li-Li \cite{Li-Li}, the tangent space $T_{\rho d\mu}\mathcal{P}_2^\infty(M, \mu)$ is identified as follows 
\begin{eqnarray*}
T_{\rho d\mu}\mathcal{P}_2^\infty(M, \mu)=\left\{s=\nabla_\mu^*(\rho \nabla\phi): \phi\in C^\infty(M), \ \ \int_M |\nabla\phi|^2\rho d\mu<\infty\right\}.
\end{eqnarray*}
For $s_i=\nabla_\mu^*(\rho\nabla\phi_i)\in T_{\rho d\mu} \mathcal{P}_2^\infty(M, \mu)$, $i=1, 2$, Otto \cite{Otto2001} introduced the following infinite dimensional Riemannian metric on $\mathcal{P}_2^\infty(M, \mu)$ 
\begin{eqnarray*}
\langle\langle s_1, s_2\rangle\rangle:=\int_M \langle \nabla \phi_1, \nabla\phi_2 \rangle \rho d\mu, 
\end{eqnarray*}
provided that 
\begin{eqnarray*}
\|s_i\|^2:=\int_M |\nabla\phi_i|^2\rho d\mu<\infty, \ \ \ i=1, 2.
\end{eqnarray*}
Let $T_{\rho d\mu}\mathcal{P}_2(M, \mu)$ be the completion of $T_{\rho d\mu}\mathcal{P}_2^\infty (M, \mu)$ with Otto's Riemannian metric. Then $\mathcal{P}_2^\infty(M, \mu)$ is a dense subspace of $\mathcal{P}_2(M, \mu)$, and 
$\mathcal{P}_2(M, \mu)$ can be regarded as a formal infinite dimensional Riemannian manifold. For more discussion about infinite dimensional Riemannian manifolds, see e.g. \cite{S.Lang,Palais}.

In \cite{BB1999}, Benamou and Brenier showed that the $L^2$-Wasserstein distance coincides with the geodesic distance between $\mu_0=\rho_0d\mu$ and $\mu_1=\rho_1d\mu$ on $\mathcal{P}_2\left(\mathbb{R}^n,dx\right)$ with Otto's infinite dimensional Riemannian metric. In the setting of weighted Riemannian manidfold $(M,g,\mu)$, the corresponding result becomes 
$$
W_2^2\left(\mu_0, \mu_1\right):=\inf \left\{\int_0^1 \int_{M}|\nabla \phi(x, t)|^2 \rho(x, t) d\mu d t: \partial_t \rho=\nabla_\mu^* (\rho \nabla \phi), \rho(0)=\rho_0, \rho(1)=\rho_1\right\} .
$$
In view of this, the geodesic flow $(\rho, \phi)$ satisfies $\eqref{Geodesic-BB}$. 

In the point of view of fluid mechanics, the Langevin deformation $(\ref{Lang-V})$ is the Newton equation of a particle moving in the phase space (i.e., the tangent bundle) over the Wasserstein space $\mathcal{P}_2(M,\mu)$ with a friction force (see \cite{Li-Li}). Indeed, let $(\rho,\phi)$ be the solution to the following system  
\begin{equation}\label{Newton-Langevin0}
\partial_t\rho=\nabla_{\mu}^*(\rho\nabla\phi),
\end{equation}
\begin{equation}\label{Newton-Langevin1}
	c^2\nabla_{\dot \rho}^{\mathcal{P}_2(M, \mu)}\dot \rho=-\nu \dot\rho-\nabla^{\mathcal{P}_2(M, \mu)} \mathcal{V},
\end{equation}
where $\nu$ is the friction coefficient of the fluid and $\nabla^{\mathcal{P}_2(M, \mu)}\mathcal{V}$ is the gradient of the internal energy functional $\mathcal{V}$ on the Wasserstein space $\mathcal{P}_2(M,\mu)$  with Otto's infinite dimensional Riemannian metric, and $\nabla_{\dot \rho}^{\mathcal{P}_2(M, \mu)}$ denotes the Levi-Civita covariant derivative along the smooth curve $t\mapsto \rho(t)d\mu$ on $\mathcal{P}_2(M,\mu)$. By \cite{Otto2001, V1, V2, OtV, Li-Li}, we have
\begin{equation*}\label{nabla V}
	\nabla^{\mathcal{P}_2(M, \mu)}\mathcal{V}=\nabla_\mu^*\left(\rho \nabla{\delta \mathcal{V}\over \delta\rho}\right).
\end{equation*}
By Lott \cite{Lott08, Lott09}, the Levi-Civita covariant derivative operator along the smooth curve $t\mapsto \rho(t)d\mu$ on the Wasserstein space $\mathcal{P}_2(M,\mu)$ is defined as follows
\begin{equation*}\label{ddot}
	\nabla_{\dot \rho}^{\mathcal{P}_2(M, \mu)}\dot \rho=\nabla_{\mu}^* \left(\rho\nabla\left({\partial\phi\over \partial t}+{1\over 2} |\nabla \phi|^2\right)\right).
\end{equation*}
Let us consider the free energy functional 
\begin{equation*}
	\mathcal{V}(\rho)=\int_M V(\rho)d\mu.
\end{equation*}
Then \eqref{Newton-Langevin1} turns to be
\begin{equation*}\label{Newton-Langevin2}
c^2\nabla_\mu^*\left(\rho\nabla\left( {\partial \phi\over \partial t}+{1\over 2} |\nabla \phi|^2\right)\right)=-\nabla_\mu^*\left(\rho\nabla \left(\nu\phi+V'(\rho)\right)\right).
\end{equation*}
That is the system of \eqref{Newton-Langevin0} and \eqref{Newton-Langevin1} is essentially equivalent to the following Langevin deformation of flows introduced in \cite{Li-Li} 
\begin{equation}\label{Langevin flow 0}
	\begin{cases}
		\partial_t\rho-\nabla_{\mu}^*(\rho\nabla\phi)=0,\\
		c^2\left(\partial_t\phi+\frac{1}{2}|\nabla\phi|^2\right)=-\phi-V^{\prime}(\rho)
	\end{cases}
\end{equation}
up to an additional constant for $\nu=1$, where $c\in[0,\infty]$,  and $V\in C^1(\mathbb R^+,\mathbb R)$. Moreover, \eqref{Langevin flow 0} is an interpolation between the Benamou-Brenier geodesic flow on the tangent bundle of the $L^2$-Wasserstein space and the gradient flow of $\mathcal{V}$. Formally, when $c=0$, we have 
$\phi=-V^\prime(\rho)$ and $(\ref{Langevin flow 0})$ becomes
\begin{equation*}\label{gradient flow 0}
\partial_t\rho=-\nabla_\mu^*\cdot(\rho\nabla V^\prime(\rho)),
\end{equation*}
which is the gradient flow of $\mathcal{V}(\rho)$ on $\mathcal{P}_2(M,\mu)$. When $c=+\infty$, in order that the second equation of $(\ref{Langevin flow 0})$ makes sense, the Hamilton-Jacobi quantity $\partial_t\phi+\frac{1}{2}|\nabla\phi|^2$ must be identically zero, i.e., it recaptures the Benamou-Brenier geodesic equations $(\ref{Geodesic-BB})$.

\subsection{Main results}\label{main section}
 Before stating our main results, let us recall the established results by Li-Li \cite{Li-Li} on the existence, uniqueness and regularity of the solution to $(\ref{Langevin flow 0})$ for any $c\in[0,\infty]$.

\begin{theorem}\label{existence}\cite{Li-Li}
Let $c \in [0, \infty]$ and $M = \mathbb R^n$ or be a compact Riemannian manifold. Let $\mathcal{V}=\int_M V(\rho)d\mu$, where $V(\rho)=\rho\log\rho$ or $V(\rho)=\frac{1}{\gamma-1}\rho^\gamma$ for $\gamma>1$. Given $(\rho_0,\phi_0) \in T\mathcal{P}_2^{\infty}(M,\mu)$ with $(\rho_0, \phi_0)\in C^{\infty}(M, \mathbb{R}^+)\times C^{\infty}(M, \mathbb{R})$, there exists $T = T_c > 0$ such that the Cauchy problem of the Langevin deformation of flows $(\ref{Langevin flow 0})$ has a unique solution $(\rho,\phi) \in C([0,T], C^\infty(M, \mathbb{R}^+)\times C^\infty(M, \mathbb{R}))$. Furthermore, let $l\geq 2$, if $(\rho_0,\varphi_0) \in H^{s+l}(M) \times H^{s+l+1}(M)$ is a small smooth initial value in the sense that there exists a sufficiently small constant $\delta_0>0$ such that $\|\rho_0 - 1\|_{s+l,2} + \|\nabla\varphi_0\|_{s+l,2} \leq \delta_0$, then $(\ref{Langevin flow 0})$ admits a unique global smooth solution $(\rho,\varphi)\in C([0,\infty), H^{s+l}(M) \times H^{s+l+1}(M))$. 
\end{theorem}


Note that, if  the potential $\mathcal{V}$ is given by the negative R\'enyi entropy, i.e., $\mathcal{V}=-\mathrm{Ent}_{\gamma}$, then the Langevin deformation becomes
\begin{equation}\label{-Lang-gamma}
\left\{
\begin{aligned}
&\partial_t \rho-\nabla_\mu^*(\rho\nabla \phi)=0,  \\
&c^2\left(\partial_t\phi+{1\over 2}|\nabla \phi|^2\right)=
-\phi+\frac{\gamma\rho^{\gamma-1}}{\gamma-1},\quad \gamma>1.
\end{aligned}
\right.
\end{equation}

Our first result is the following theorem on the monotonicity and convexity of the Hamiltonian and the Lagrangian along the solution. It has its own interest. 

\begin{theorem}\label{MT0}   Let $(M, g)$  be the Euclidean space $\mathbb{R}^n$ or a compact Riemannian manifold, $f\in C^2(M)$ and $d\mu=e^{-f}dv$. 
For any $c\in (0, \infty)$, let $(\rho(t), \phi(t))$ be a smooth solution to \eqref{-Lang-gamma}.  Let
\begin{eqnarray*}
H(\rho, \phi)&=&{c^2\over 2}\int_M |\nabla\phi|^2\rho\hspace{0.2mm} d\mu
-{1\over \gamma-1}\int_{M}\rho^\gamma\hspace{0.2mm} d\mu,\\
L(\rho, \phi)&=&{c^2\over 2}\int_M |\nabla\phi|^2\rho\hspace{0.2mm} d\mu+{1\over \gamma-1}
\int_{M}\rho^\gamma\hspace{0.2mm} d\mu.
\end{eqnarray*}
Then
\begin{eqnarray*}
{d\over dt}H(\rho(t), \phi(t))&=&-\int_M |\nabla \phi|^2\rho\hspace{0.2mm} d\mu,\label{monotoH}\\
{d^2\over dt^2}L(\rho(t), \phi(t))&=&2\int_M \left[c^{-2}|\nabla\phi+\gamma\rho^{\gamma-2}\nabla\rho|^2\rho\right]d\mu\\
&&+2\int_M\left[\|\nabla^2\phi\|_{\mathrm{HS}}^2+\mathrm{Ric}(L)(\nabla\phi, \nabla\phi)+(\gamma-1)(L\phi)^2\right] \rho^\gamma d\mu. 
\end{eqnarray*}
In particular, the Hamiltonian $H(\rho(t), \phi(t))$ is  nonincreasing, and if the CD$(0, \infty)$-condition holds, then the Lagrangian $L(\rho(t), \phi(t))$ is convex along the Langevin deformation  $(\rho(t), \phi(t))$ defined by $\eqref{-Lang-gamma}$. 
\end{theorem}

The second result is the $W$-entropy-information formula for the Langevin deformation $(\ref{Lang-gamma})$ on the $L^2$-Wasserstein space over compact or complete  Riemannian manifolds. 

\begin{theorem}\label{W entropy of Langevin flow}
Let $c\in(0,+\infty)$ and $M$ be a compact Riemannian manifold or a complete Riemannian manifold with bounded geometry condition\footnote{Here, we say that $(M,g)$ satisfies the bounded geometry condition if the Riemannian curvature tensor $\mathrm{Riem}$ and its covariant derivatives $\nabla^k \mathrm{Riem}$ are uniformly bounded on $M$ for $k = 1, 2, 3$.}. Let $(\rho,\phi): M\times[0,T]\to\mathbb R^+\times\mathbb R$ be a smooth solution with suitable growth condition\footnote{For the exact description of the suitable growth condition, see Theorem \ref{complete mfd}.} to the Langevin deformation $\eqref{Lang-gamma}$ on $T\mathcal{P}_2^\infty(M, \mu)$. Let $\rho_{c,m}$ be the reference model given by \eqref{rho_c,m}. For $m\geq n$, define 
\begin{equation*}
  H_{c,m}(\rho)=\mathrm{Ent}_\gamma(\rho)-\mathrm{Ent}_\gamma(\rho_{c,m}), 
\end{equation*}
and define the $W$-entropy for the Langevin deformation of flows for R\'enyi entropy by
\begin{equation*}\label{W_cm}
  W_{c,m}(\rho,t)=a(t)H_{c,m}(\rho)+b(t)\frac{d}{dt}H_{c,m}(\rho),
\end{equation*}
where $a(t)$ and $b(t)$ satisfy the following equations
\begin{eqnarray}\label{eq of ab}
\frac{a(t)+ b^\prime(t)}{b(t)}&=&2\alpha(t)\left[m(\gamma-1)+1\right]+\frac{1}{c^2},\\
\frac{a^\prime(t)}{b(t)}&=&\left((\gamma-1)m+(\gamma-1)^2m^2\right)\alpha^2(t),	
\end{eqnarray}
 where $\alpha=\frac{u^\prime}{u}$ and $u$ is a solution to the following ODE on interval $[\delta,T]\subset(0,+\infty)$
\begin{equation}\label{equation of u}
	c^2u^{\prime\prime}(t)+u^{\prime}(t)=ku^{1-\frac{1}{k}}(t),
\end{equation}
with given initial data $u(\delta) > 0$ and $u^\prime(\delta) \in \mathbb R$ for any $\delta > 0$, where 
\begin{equation}\label{k}
	k=\frac{1}{m(\gamma-1)+2}.
\end{equation}
Define the relative Fisher information by 
\begin{equation*}
  I_{c,m}(\rho)=\|\nabla \mathrm{Ent}_\gamma(\rho)\|^2-\|\nabla \mathrm{Ent}_\gamma(\rho_{c,m})\|^2,  
\end{equation*}
where $\nabla{\mathrm{Ent}_\gamma}$ is the gradient of $\mathrm{Ent}_\gamma$ with respect to Otto's infinite dimensional Riemannian metric on $\mathcal{P}_2^{\infty}(M,\mu)$. Then the following $W$-entropy-information formula holds
\begin{equation}\label{W entropy for L flow}
\begin{aligned}
&\frac{1}{b(t)}\frac{d}{dt}W_{c,m}(\rho(t),t)+\frac{1}{c^2}I_{c,m}(\rho(t))\\
&\hskip1cm =\int_M \rho^\gamma\left[\left\|\nabla^2\phi-\alpha(t) g\right\|_{\mathrm{HS}}^2+\mathrm{Ric}_{m,n}(L)(\nabla\phi,\nabla\phi)\right.\\
&\hskip1.2cm \left.+(\gamma-1)(L\phi-m\alpha(t))^2+(m-n)\left(\frac{\nabla\phi\cdot\nabla f}{m-n}+\alpha(t)\right)^2\right]d\mu.
\end{aligned}
\end{equation}  
In particular, if $\mathrm{Ric}_{m,n}(L)\geq 0$, then for all $t>0$, we have the following $W$-entropy-information comparison inequality
\begin{equation}
  \frac{1}{b(t)}\frac{d}{dt}W_{c,m}(\rho(t),t)+\frac{1}{c^2}I_{c,m}(\rho(t))\geq 0.
  \label{WI-CI}
\end{equation}

\end{theorem}

In Section  $\ref{ref model sec}$, for $m\in \mathbb{N}$, we will give the explicit construction of the reference model $(\rho_{c,m},\phi_{c,m})$ to the  Langevin deformation of flows $(\ref{Lang-gamma})$ on $T\mathcal{P}_2^\infty(\mathbb R^m,dx)$. We can prove that 
\begin{equation}
  \frac{1}{b(t)}\frac{d}{dt}W_{c,m}(\rho_{c,m}(t),t)+\frac{1}{c^2}I_{c,m}(\rho_{c,m}(t))=0.\label{WI-CI2}
\end{equation}
Moreover, we will prove that $(M, \rho, \phi)= (\mathbb{R}^m,\rho_{c,m},\phi_{c,m})$ is 
the rigidity model for the $W$-entropy-information inequality $(\ref{WI-CI})$ for the Langevin deformation of flows $(\ref{Lang-gamma})$ on $T\mathcal{P}_2^\infty(M, \mu)$ over a complete weighted Riemannian manifold $(M, g, \mu)$ with bounded geometry condition and  CD$(0, m)$-condition.  In particular, when $m=n$, $(M, \rho, \phi)= (\mathbb{R}^n,\rho_{c, n},\phi_{c, n})$ is the  rigidity model for the $W$-entropy-information inequality $(\ref{WI-CI})$ on complete Riemannian manifold with bounded geometry condition and with non-negative Ricci curvature.

In the case of the porous medium equation, i.e., $c=0$, the $W$-entropy formula for the standard porous medium equation $\partial_t u=\Delta u^\gamma$ was first obtained by Lu-Ni-Vazquez-Villani \cite{LNVV}. We prove the rigidity theorem for this $W$-entropy. In this case, we can prove that
\begin{equation*}
	a(t)=(2-2k)t^{1-2k} \ \ {\rm and}\ \ b(t)=t^{2-2k}
\end{equation*}
satisfy  $\eqref{eq of ab}$, and we have the following

\begin{theorem}\label{W entropy for gradient flow}
Let $(M, g, \mu)$ be a weighted compact or a complete Riemannian manifold with bounded geometry condition. Let $\rho$ be a positive and smooth solution to the porous medium equation
\begin{equation}
	\partial_tu=L u^\gamma, \quad \gamma>1, \label{PMEL}
\end{equation}
Define the $H_{0,m}$-entropy and $W_{0,m}$-entropy as
\begin{equation*}
	H_{0,m}(\rho)=\mathrm{Ent}_\gamma(\rho)-\mathrm{Ent}_\gamma(\rho_{0,m}),
\end{equation*}
\begin{equation*}
	W_{0,m}(\rho,t)=\frac{d}{dt}\left(t^{2-2k}H_{0,m}(\rho(t))\right),
\end{equation*}
where $\rho_{0,m}$ is given by \eqref{rho_0,m} and $k$ is given by \eqref{k}. Then 
\begin{equation}\label{W entropy eq for gra}
\begin{aligned}
\frac{d}{dt}W_{0,m}(\rho,t)=&
	2t^{2-2k}\int_M \rho^\gamma\left[\left\|\nabla^2\phi-\frac{kg}{t}\right\|_{\mathrm{HS}}^2+\mathrm{Ric}_{m,n}(L)(\nabla\phi,\nabla\phi)\right.\\
	&\qquad \left.+(\gamma-1)\left(L\phi-\frac{mk}{t}\right)^2+(m-n)\left(\frac{\nabla\phi\cdot\nabla f}{m-n}+\frac{k}{t}\right)^2\right]d\mu,
\end{aligned}
\end{equation}
where 
\begin{equation*}
	\phi=-\frac{\gamma}{\gamma-1}\rho^{\gamma-1}.
\end{equation*}
In particular, if $\mathrm{Ric}_{m,n}(L) \geq 0$, then $\frac{d}{dt}W_{0,m}(\rho,t)\geq0$ for all $t>0$. Moreover, under the condition $\mathrm{Ric}_{m, n}(L)\geq 0$, $\frac{d}{dt}W_{0,m}(\rho,t)=0$ holds at some $t=t_0 >0$ if and only if $(M,g)$ is isomeric to $\mathbb R^n$, $n = m$, $f = C$ is a constant, and $\rho=\rho_{0,m}$. 
\end{theorem}


In the case of geodesic flow, i.e., $c=+\infty$, the $W$-entropy formula and rigidity theorem were first
 proved by Li-Li \cite{Li-Li}. Here we can use the R\'enyi entropy to obtain another $W$-entropy formula for the geodesic flow on the Wasserstein space over compact or a complete Riemannian manifold with bounded geometry condition. In this case, we can prove that 
\begin{equation*}
	a(t)=\left(\frac{1}{k}-1\right)t^{\frac{1}{k}-2} \ \ {\rm and}\ \ b(t)=t^{\frac{1}{k}-1}
\end{equation*}
satisfy $\eqref{eq of ab}$, and we have the following 

\begin{theorem}\label{W entropy for geodesic}
Let $(M,g)$	be a compact or a complete Riemannian manifold with  bounded geometry condition. Let $(\rho,\phi):M\times[0,T]\to\mathbb R^+\times\mathbb R$ be a smooth sulution with suitable growth condition to the geodesic equations $\eqref{Geodesic-BB}$ on $T\mathcal{P}_2^\infty(M, \mu)$. For $m\geq n$, define the $H_{\infty,m}$-entropy and the $W_{\infty,m}$-entropy respectively as follows
\begin{equation*}
	H_{\infty,m}(\rho)=\mathrm{Ent}_\gamma(\rho)-\mathrm{Ent}_\gamma(\rho_{\infty,m}),
\end{equation*}
\begin{equation*}
	W_{\infty,m}(\rho,t)=\frac{d}{dt}\left(t^{\frac{1}{k}-1}H_{\infty,m}(\rho(t))\right),
\end{equation*}
where $\rho_{\infty,m}$ is given by \eqref{rho_infty,m}, and $k$ is given by \eqref{k}. Then 
\begin{equation}
\begin{aligned}
\frac{d}{dt}W_{\infty,m}(\rho,t)=&
	t^{\frac{1}{k}-1}\int_M \rho^\gamma\left[\left\|\nabla^2\phi-\frac{g}{t}\right\|_{\mathrm{HS}}^2+\mathrm{Ric}_{m,n}(L)(\nabla\phi,\nabla\phi)\right.\\
	&\qquad \left.+(\gamma-1)\left(L\phi-\frac{m}{t}\right)^2+(m-n)\left(\frac{\nabla\phi\cdot\nabla f}{m-n}+\frac{1}{t}\right)^2\right]d\mu.
\end{aligned}
\end{equation}
In particular, if ${\mathrm Ric}_{m, n}(L) \geq 0$, then 
$\frac {d} {dt} W_{\infty, m}(\rho,t)\geq 0$ for all $t>0$. Moreover, under the condition $\mathrm{Ric}_{m, n}(L)\geq 0$, $\frac {d} {dt}W_{\infty, m}(\rho, t) = 0$ holds at some $t=t_0 >0$ if and only if $(M, g)$ is isomeric to $\mathbb R^n$, $n = m$, $f = C$ is a constant, and $(\rho,\phi) = (\rho_{\infty,m},\phi_{\infty,m})$.

\end{theorem}

Theorem \ref{W entropy for gradient flow} and Theorem \ref{W entropy for geodesic} indicate that, in the sense of statistical mechanics, among all solutions to the porous medium equation and geodesic flows on all complete weighted Riemannian manifolds with bounded geometry condition and  CD$(0, m)$-condition, the special solutions on the  Euclidean space $(\mathbb{R}^n,\rho_{0, n},\phi_{0, n})$ and $(\mathbb{R}^n,\rho_{\infty, n},\phi_{\infty, n})$ are the unique minimum point (i.e., the equilibrium state) of the corresponding $W$-entropy respectively, regarded as the Hamiltonian energy functional, while Theorem \ref{W entropy of Langevin flow} indicates that among all solutions to the Langevin deformation between the Benamou-Brenier geodesic flow and the porous medium equation on complete weighted Riemannian manifolds with bounded geometry condition and CD$(0, m)$-condition, the special solution $(\mathbb R^n, \rho_{c, n}, \phi_{c, n})$ is the unique one (i.e., the equilibrium state) such that the inequality in \eqref{WI-CI} becomes the equality, i.e., \eqref{WI-CI2}. For more detail, see Section \ref{section of cor}. 

As we mentioned before, in the extremal cases $c=0$ and $c=+\infty$, the Langevin deformation of flows $(\ref{Lang-gamma})$ correspond to the gradient flow of the R\'enyi entropy (i.e., the porous medium equation) and the Benamou-Brenier geodesic flow respectively. We will rigorously prove this convergence result  in Section 5, see Theorem \ref{main.convergence}. 

\begin{remark}
It is well-known that the Boltzmann-Shannon entropy is the limit of the modified R\'enyi entropy as $\gamma\to1$, i.e. 
\begin{equation*}
	\int_M \rho\log\rho d\mu=\lim_{\gamma\to1}\frac{1}{\gamma-1}\left(\int_M \rho^\gamma d\mu-1\right).
\end{equation*}
In view of this, it is reasonable to introduce the Langevin deformation for $\overline{\mathrm{Ent}}_\gamma$
defined as follows
\begin{eqnarray*}
	\overline{\mathrm{Ent}}_\gamma(\rho)=\frac{1}{\gamma-1}\left(\int_M \rho^\gamma dx-1\right).
\end{eqnarray*}
Note that the solution $(\bar \rho,\bar\varphi)$ to the Langevin deformation for $\overline{\mathrm{Ent}}_\gamma$ is related to the solution $(\rho,\varphi)$ to the Langevin deformation for $\mathrm{Ent}_\gamma$ by
\begin{eqnarray}(\bar \rho,\bar\varphi)=(\rho,\varphi+\frac{1}{\gamma-1}). \label{two}\end{eqnarray}
The relative entropies for the Langevin deformation for $\overline{\mathrm{Ent}}_\gamma$ and the Langevin deformation for $\mathrm{Ent}_\gamma$ satisfy
\begin{eqnarray*}
H_{c,m}(\bar\rho(t), \bar\phi(t))&=&\overline{{\rm Ent}}_\gamma(\rho(t))-\overline{{\rm Ent}}_\gamma(\rho_{c, m}(t))\\
&=&{\rm Ent}_\gamma(\rho(t))-{\rm Ent}_\gamma(\rho_{c, m}(t))\\
&=&H_{c,m}(\rho(t), \phi(t)).
\end{eqnarray*}
This yields that the $W$-entropy-information formula for the Langevin deformation for $\overline{\mathrm{Ent}}_\gamma$ is the same as  the $W$-entropy-information formula for the Langevin deformation for $\mathrm{Ent}_\gamma$, and the corresponding rigidity models are related by $(\ref{two})$. 

\end{remark}

\section{Variational formulae for R\'enyi entropy}

In this section, we prove some variational formulae for the R\'enyi entropy on the Wasserstein space over a Riemannian manifold with natural curvature-dimension condition. 

First we recall two results on the Langevin deformation of flows on the tangent bundle over a Riemannian manifold. They were proved in the 2021 arxiv version of Li-Li \cite{Li-Li} but were not included in the published version of \cite{Li-Li} for saving the length of the paper. We reproduce them here. 
\begin{proposition}\label{propLL1}\cite{Li-Li}
Let $M$ be a Riemannian manifold and $TM$ be the tangent bundle of $M$. For any $c>0$, let $\left(x_t, v_t\right)$ be the Langevin deformation of flows on $T M$ defined by
\begin{equation*}
\begin{cases}
\dot x= v,\\
c^2\dot v= -v-\nabla V(x),	
\end{cases}
\end{equation*}
where (with a slight abuse of notations) $V\in C^2(M)$ is the potential on $M$. Then
\begin{equation*}
\frac{d^2}{d t^2} V(x)+c^{-2} \frac{d}{d t} V(x)+c^{-2}|\nabla V|^2=\nabla^2 V(v, v).
\end{equation*}
In particular, if $\nabla^2 V \geq K$, where $K \in \mathbb{R}$ is a constant, then we have
$$
\frac{d^2}{d t^2} V(x)+c^{-2} \frac{d}{d t} V(x)+c^{-2}|\nabla V|^2 \geq K|v|^2.
$$
\end{proposition}

\begin{proof}
Indeed, a simple calculation yields
$$
\begin{aligned}
\frac{d}{d t} V(x(t)) & =\langle\nabla V, \dot{x}\rangle, \\
\frac{d^2}{d^2 t} V(x(t)) & =\nabla^2 V(\dot{x}, \dot{x})+\langle\nabla V, \ddot{x}\rangle \\
& =\nabla^2 V(v, v)-c^{-2}\langle\nabla V, v+\nabla V\rangle.
\end{aligned}
$$
Hence
$$
\frac{d^2}{d t^2} V(x)+c^{-2} \frac{d}{d t} V(x)+c^{-2}|\nabla V|^2=\nabla^2 V(v, v) \geq K|v|^2.
$$
This finishes the proof.	
\end{proof}
Applying the Otto calculus, the above result can be extended to the Langevin deformation of flows on the tangent bundle over the $L^2$-Wasserstein space. 

\begin{theorem}\cite{Li-Li}
Let $(\rho,\phi)$ be a smooth solution to the Langevin deformation $(\ref{Langevin flow 0})$ on $T\mathcal{P}_2(M,\mu)$, where the potential $V\in C^2(\mathbb R^+,\mathbb R)$.  
\begin{itemize}
	\item If $c\in(0,+\infty)$, then we have
	\begin{equation}\label{ineqV}
		{d^2\over dt^2}{\mathcal{V}}+c^{-2}{d\over dt} {\mathcal{V}}+c^{-2}\|\nabla \mathcal{V}\|^2=\nabla^2\mathcal{V}(\rho)(\dot\rho,\dot\rho)=\left\langle\langle\nabla_{\dot\rho}\nabla \mathcal{V},\dot\rho\right\rangle\rangle.
	\end{equation}
	\item If $c=+\infty$, i.e., $(\rho,\phi)$ is a smooth solution to the Benamou-Brenier geodesic equation $\eqref{Geodesic-BB}$, then 
	\begin{equation}\label{2 varitaion for geo}
	\frac{d^2}{dt^2}{\mathcal{V}}(\rho(t))=\nabla^2 \mathcal{V}(\dot\rho,\dot\rho)=\left \langle\langle\nabla_{\dot\rho}\nabla \mathcal{V},\dot\rho\right\rangle\rangle.
	\end{equation}
	\item If $c=0$, i.e., $(\rho,\phi)$ is a smooth solution to the gradient flow equation $\partial_t\rho=-\nabla \mathcal{V}(\rho)$, then 
	\begin{equation}\label{2 varitaion for gra}
	\frac{d^2}{dt^2} \mathcal{V}(\rho(t))=2\nabla^2\mathcal{V}(\rho)(\dot\rho,\dot\rho)=2\left\langle\langle\nabla_{\dot\rho}\nabla \mathcal{V},\dot\rho\right\rangle\rangle.
	\end{equation}
\end{itemize}
Here $\nabla^2\mathcal{V}(\rho)(\dot\rho,\dot\rho)$ denotes the Hessian of $\mathcal{V}(\rho)$ 
with respect to Otto's infinite dimensional metric on 
$\mathcal{P}_2^\infty(M,\mu)$.

\end{theorem}
\begin{proof} The proof  is similar to the one of Proposition \ref{propLL1}. 
Here we only give the proof to the case of $c=0$. Indeed, 
\begin{equation*}
\frac{d^2}{dt^2} \mathcal{V}(\rho(t))=\frac{d}{dt}\langle\langle\nabla \mathcal{V},\dot\rho\rangle\rangle=\left\langle\langle\nabla_{\dot\rho}\nabla \mathcal{V},\dot\rho\right\rangle\rangle+\left\langle\langle\nabla \mathcal{V},\nabla_{\dot\rho}\dot\rho\right\rangle\rangle.
\end{equation*}	
Since $\dot\rho=-\nabla \mathcal{V}(\rho)$, we have 
\begin{equation*}
\frac{d^2}{dt^2} \mathcal{V}(\rho(t))=2\langle\left\langle\nabla_{\dot\rho}\nabla \mathcal{V},\dot\rho\right\rangle\rangle=2\nabla^2\mathcal{V}(\rho)(\dot\rho,\dot\rho).	
\end{equation*}
\end{proof}

The following result gives the Hessian of $\mathcal{V}$ on $\mathcal{P}_2^\infty(M,\mu)$.
%
\begin{proposition} \cite{Lott-Villani, V1, V2, Li-Li}\label{HessV}
Let $(M,\mu)$ be a compact Riemannian manifold or a complete Riemannian manifold with bounded geometry condition. Let $\mathcal{V}(\rho)=\int_M V(\rho)d\mu$ with $V\in C^2(\mathbb R^+,\mathbb R)$. Let $\rho=(\rho(t))$ be a smooth curve on $\mathcal{P}_2^\infty(M,\mu)$ with suitable growth condition and $\partial_t\rho=\nabla_\mu^*(\rho\nabla\phi)$. Then we have 
\begin{equation}\label{Hess of V}
\begin{aligned}
	\nabla^2\mathcal{V}(\rho)(\dot\rho,\dot\rho)=&\int_M \left[P_2(\rho)(L\phi)^2+P(\rho)\left(\|\nabla^2\phi\|_{\mathrm{HS}}^2+\mathrm{Ric}(L)(\nabla\phi,\nabla\phi)\right)\right]d\mu,
\end{aligned}
\end{equation}	
where 
\begin{equation*}
	P(\rho)=\rho V^\prime(\rho)-V(\rho), 
\end{equation*}
\begin{equation*}
	P_2(\rho)=\rho P^\prime(\rho)-P(\rho)=\rho^2V^{\prime\prime}(\rho)-\rho V^{\prime}(\rho)+V(\rho).
\end{equation*}
\end{proposition}
\begin{proof} See \cite{Lott-Villani, V1, V2} and Li-Li \cite{Li-Li}. \end{proof}


Combining $(\ref{ineqV})$ with Proposition \ref{HessV}, we have the following result which holds for the Langevin deformation of flows \eqref{Langevin flow 0} with  general potential on the $L^2$-Wasserstein space over a Riemannian manifold. 

\begin{theorem}\label{th1}\label{SecondVIn}
Let $(M,\mu)$ be a compact Riemannian manidfold and $(\rho,\phi)$ be a smooth solution of the Langevin deformation of flows $(\ref{Langevin flow 0})$ with $V\in C^2(\mathbb R^+,\mathbb R)$. Then we have 
\begin{equation}\label{W entrop-1}
\begin{aligned}
	&{d^2\over dt^2} {\mathcal{V}}+\frac{1}{c^2} {d\over dt} {\mathcal{V}}+\frac{1}{c^2}\| \nabla\mathcal{V}(\rho)\|^2\\
	=&\int_M \left[P_2(\rho)(L\phi)^2+P(\rho)\left(\|\nabla^2\phi\|^2_{\mathrm{HS}}+\mathrm{Ric}(L)(\nabla\phi,\nabla\phi)\right)\right]d\mu,
\end{aligned}
\end{equation}
where  $\nabla \mathcal{V}$ is the gradient of $\mathcal{V}$ on $\mathcal{P}_2^{\infty}(M,\mu)$. In particular, assume that $\mathrm{Ric}_{m,n}(L)\geq K$ and 
\begin{equation}\label{cond1}
P_2(\rho)\geq -\frac{1}{m}P(\rho), 
\end{equation}
\begin{equation}\label{cond2}
P(\rho)\geq 0. 
\end{equation}
Then for all $c>0$, we have 
\begin{equation}\label{ineq 1}
	{d^2\over dt^2} {\mathcal{V}}+\frac{1}{c^2}{d\over dt} {\mathcal{V}}+
	\frac{1}{c^2}\|\nabla \mathcal{V}(\rho)\|^2
	\geq K\int_M \left(\rho V^{\prime}(\rho)-V(\rho)\right)|\nabla\phi|^2  d\mu.
\end{equation}
\end{theorem}
\begin{proof} 
Combining $(\ref{W entrop-1})$ and $(\ref{cond1})$, we have 
\begin{equation*}
 \nabla^2\mathcal{V}(\dot\rho,\dot\rho)\geq \int_{\mathbb R^n} \left(\rho V^{\prime}(\rho)-V(\rho)\right)\left(-\frac{1}{m}(L\phi)^2+\|\nabla^2\phi\|^2_{\mathrm{HS}}+\mathrm{Ric}(L)(\nabla\phi,\nabla\phi)\right)d\mu.
\end{equation*} 
Note that (see e.g. \cite{BE, Lott03, Li-2005})
\begin{equation*}
  \mathrm{Ric}_{m,n}(L)\geq K\Longleftrightarrow \Gamma_2(\phi,\phi)\geq \frac{1}{m}(L\phi)^2+K|\nabla \phi|^2,
\end{equation*}
where
\begin{equation*}
  \Gamma_2(\phi,\phi)=\|\nabla^2 \phi\|_{\mathrm{HS}}^2+\mathrm{Ric}(L)(\nabla\phi,\nabla\phi).
\end{equation*}
Thus, under the assumptions $(\ref{cond1})$, $(\ref{cond2})$ and the CD$(K,m)$ condition (i.e., $\mathrm{Ric}_{m,n}(L)\geq K$), we have
\begin{equation*}
 \nabla^2\mathcal{V}(\dot\rho,\dot\rho)\geq K\int (\rho V^{\prime}(\rho)-V(\rho))|\nabla\phi|^2 d\mu.
\end{equation*}
\end{proof}

\begin{remark}
If the solution $(\rho,\phi)$ to the Langevin deformation of flows $(\ref{Langevin flow 0})$ with general potential $V\in C^2(\mathbb R^+,\mathbb R)$ satisfies the growth condition of \cite[Therem 4.1]{Li2024}, then, by \cite[Therem 4.1]{Li2024}, the variational formula \eqref{W entrop-1} can be extended to the $L^2$-Wasserstein space over complete non-compact Riemannian manifolds with bounded geometry condition. 
\end{remark}

\begin{remark}
McCann \cite{McCann1997} gave the criteria for geodesically displacement convexity for the internal energy. Sturm and von Renesse \cite{Sturm-Renesse} proved the equivalence between the $K$-geodesical convexity of the Boltzmann-Shannon entropy $\mathrm{Ent}$ on $\mathcal{P}_2(M,\mu)$ and the CD$(K,\infty)$-condition. By \cite{Sturm06-1,Lott-Villani}, the geodesical convexity of the R\'enyi entropy $\mathrm{Ent}_\gamma$ with $\gamma<1$ on $\mathcal{P}_2(M,\mu)$ is equivalent to the CD$(K,N)$-condition where $N\leq \frac{1}{1-\gamma}$. In the case $\mathcal{V}=\mathrm{Ent}$ or $\mathcal{V}=\mathrm{Ent}_\gamma$ with $\gamma>1$, $\eqref{cond1}$ and $\eqref{cond2}$ hold for all $m\in\mathbb R$. In the case $\mathcal{V}=\mathrm{Ent}_\gamma$ with $\gamma<1$, $\eqref{cond1}$ and $\eqref{cond2}$ hold for all $m\in(0,\frac{1}{1-\gamma}]$. 
\end{remark}

Next we extend above entropy variational formulae to complete non-compact Riemannian manifolds. To this end, we only need to justify the exchangeability of the order of time derivative(s) and integration, for which we  use the Lebesgue dominated convergence theorem. In particular, we consider $V(\rho)=\frac{1}{\gamma-1}\rho^\gamma$ and assume that $\rho$ and $\phi$ satisfy some integrability conditions, then we obtain the following  

\begin{theorem}\label{complete mfd}
Let $(M,\mu)$ be a weighted complete Riemannian manifold, where $\mu=e^{-f}dv$ and $f\in C^2(M)$. Suppose that $\mathrm{Ric}(L)=\mathrm{Ric} + \nabla^2 f$ is uniformly bounded on $M$, i.e., there exists a constant $C > 0$ such that $|\mathrm{Ric}(L)| \leq C$. Let $(\rho,\phi)$ be a smooth solution to the geodesic equation \eqref{Geodesic-BB}. Assume that the following integrability condition holds
\begin{equation*}
	\int_M \rho^\gamma\left[\|\nabla^2\phi\|_{\mathrm{HS}}^2+|\nabla\phi|^2+(L\phi)^2+|\nabla L\phi|^2\right]d\mu<+\infty,
\end{equation*}
\begin{equation*}
	\int_M |\nabla\phi|^2\rho^\gamma d\mu<+\infty,
\end{equation*}
\begin{equation*}
	\mathrm{Ent}_\gamma(\rho)=\frac{1}{\gamma-1}\int_M \rho^\gamma d\mu<+\infty,
\end{equation*}
and
\begin{equation*}
	I_\gamma(\rho)=\|\nabla \mathrm{Ent}_\gamma(\rho)\|^2=\int_M \frac{|\nabla\rho^\gamma|^2}{\rho}d\mu<+\infty.
\end{equation*}
Then 
\begin{eqnarray*}
	\frac{d}{dt}\mathrm{Ent}_\gamma(\rho(t))&=&\int_M\langle\nabla\rho^\gamma,\nabla\phi\rangle d\mu=-\int_M L\phi\rho^\gamma d\mu,\label{1 derivative}\\
	\frac{d^2}{dt^2}\mathrm{Ent}_\gamma(\rho(t))&=&\int_M \rho^\gamma\left[(\gamma-1)(L\phi)^2+\|\nabla^2\phi\|_{\mathrm{HS}}^2+\mathrm{Ric}(L)(\nabla\phi,\nabla\phi)\right] d\mu.\label{2 derivative}
\end{eqnarray*}	
\end{theorem}

\begin{proof} 
Since $M$ is a complete Riemannian manifold, we can choose an increasing sequence of functions $\{\eta_k\}$ in $C_0^{\infty}(M)$ such that $0 \leq \eta_k \leq 1, \eta_k=1$ on $B(o, k), \eta_k=0$ on $M \backslash B(o, 2 k)$, and $\left|\nabla \eta_k\right| \leq \frac{1}{k}$. Based on the Lebesgue dominated convergence theorem and by standard cut-off argument and integration by parts for compactly supported functions $\eta_k \rho$, we can prove this result. To save length of the paper, we omit the detail here, which can be referred to \cite[Theorem 3.1]{Li-Li} for $\gamma=1$. See also Li \cite[Theorem 4.1]{Li2024} for the entropy dissipation formulae for the general potential functional $\mathcal{V}=\int_M V(\rho(x))d\mu(x)$ along the Benamou-Brenier geodesic on the $L^2$-Wasserstein space over complete Riemannian manifolds.

\end{proof}

\begin{remark}
Let $(M,g,\mu)$ be a complete Riemannian manifold with $|\mathrm{Ric}(L)|\leq C$. Assume that the solutions to the geodesic equation $(\ref{Geodesic-BB})$ and the Langevin deformation of flows $(\ref{Lang-gamma})$ satisfy suitable growth conditions such that the integrability conditions in Theorem $\ref{complete mfd}$ hold, then we can extend the $W$-entropy(-information) formulae for the geodesic flow and the Langevin deformation of flows to complete non-compact Riemannian manifolds. 

As pointed out by Remark 5.7 in \cite{LNVV}, the $W$-entropy formula for porous medium equation can be proved for complete noncompact Riemannian manifolds with the help of Aronson–Bénilan estimates obtained in \cite{LNVV}. On the other hand, we can use the property that the solution to the porous medium equation has a finite propagation speed to guarantee the integration by parts. 

\end{remark}

Next we prove that the solution to porous medium equation $\partial_t u=L u^\gamma$ has a finite propagation speed if $(M,\mu)$ satisfies CD$(0,m)$-condition. In the case of standard porous medium equation
\begin{equation}\label{porous medium}
	\partial_tu=\Delta u^\gamma, \quad \gamma>1, 
\end{equation}
the property of finite propagation speed for solutions of $(\ref{porous medium})$ 
was proved by Vázquez \cite{Vazquez2015} 
 on hyperbolic spaces, and  by Grillo and Muratori \cite{Grillo} on Cartan-Hadamard manifolds. 
 Recently,  Grigor’yan-S\"urig \cite{Grigoryan} proved the property of  finite propagation speed of 
 weak subsolution to $(\ref{porous medium})$ on complete non-compact Riemannian manifold 
 with non-negative Ricci curvature. Using the result of Grigor’yan-S\"urig \cite{Grigoryan}, we can 
 prove the following theorem, which has its own interests. 


\begin{theorem}
Let $(M,g,e^{-f}dv)$ be a weighted complete non-compact Riemannian manifold with bounded geometry 
condition. Let u be a bounded non-negative solution to the weighted porous medium equation 
\begin{equation*}
	\partial_tu=L u^\gamma, \quad \gamma>1, \label{PMEL}
\end{equation*}
on $M\times\mathbb R^+$ with the initial condition $u(\cdot,0)=u_0$. 
Assume that $\mathrm{Ric}_{m,n}(L)\geq 0$ and $u_0$ has a compact support $K\subset M$. Then, for any $t\geq 0$, it holds
\begin{equation*}
	\mathrm{supp}~u(\cdot,t)\subset K_{Ct^{\frac{1}{2}}},
\end{equation*}
where $C=C(\|u_0\|_{L^{\infty}},\gamma,n)$ and $K_{Ct^{\frac{1}{2}}}$ is a closed $Ct^{\frac{1}{2}}$-neighborhood of K, i.e 
\begin{equation*}
	K_{Ct^{\frac{1}{2}}}=\left\{x\in M: d(x,K)\leq Ct^{\frac{1}{2}}\right\}.
\end{equation*}
\end{theorem}

\begin{proof}
By \cite{Lott03,Li-2005}, the $m$-dimensional  Bakry-Emery curvature $\mathrm{Ric}_{m,n}(L)$ on $M$ is the horizantal component of the Ricci curvature on $M\times N$ equipped with a suitable warped product. More precisely, let $(M,g^M)$ and $(N,g^N)$ be an $n$-dimensional and an $(m-n)$-dimensional Riemannian manifold, respectively. 
Denote $\overline{M}=M\times N$. The warped Riemannian metric on $\bar{M}$ is given by 
\begin{equation*}
	g^{\overline M}=g^M\otimes e^{-\frac{2f}{m-n}}g^N,
\end{equation*}
where $f\in C^2(M)$. By Li-Li \cite{Li-Li2015}, the Laplace-Beltrami on $(\bar{M}, g^{\overline M})$ is given by
$$\overline{\Delta}_{\overline M}=L+e^{-\frac{2f}{m-n}}\Delta_N.$$ 
If $u\in C\left([0,T],C^2(M, \mathbb{R})\right)$ is a bounded non-negative solution to the weighted porous medium equation \eqref{PMEL} on $M$, then $u$ is a bounded non-negative solution to 
\begin{equation*}
	\partial_t u=\overline{\Delta}_{\overline M} u^\gamma,\quad \gamma>1
\end{equation*}
on $\overline M$. On the other hand, a direct calculation yields \cite{Besse1987, Lott03, Li-Li2015}
\begin{equation*}
\begin{aligned}
\mathrm{\overline{Ric}}_{\overline M}&=\left(
\begin{array}{cc}
\mathrm{Ric}_M+\nabla^2f-\frac{\nabla f\otimes\nabla f}{m-n} & 0\\
0 & \mathrm{Ric}_N+\left[\frac{\Delta f}{m-n}+(m-n-2)\frac{|\nabla f|^2}{(m-n)^2}\right]
\end{array}
\right)\\
&=\left(\begin{array}{cc}
\mathrm{Ric}_{m,n}(L) & 0\\
0 & \mathrm{Ric}_N+\left[\frac{\Delta f}{m-n}+(m-n-2)\frac{|\nabla f|^2}{(m-n)^2}\right]
\end{array}
\right).
\end{aligned}
\end{equation*}
Since $f$ satisfies the bounded geometry condition, there exists a complete Riemannain manifold  $N$ such that 
\begin{equation*}
\mathrm{Ric}_N+\left[\frac{\Delta f}{m-n}+(m-n-2)\frac{|\nabla f|^2}{(m-n)^2}\right]\geq 0.
\end{equation*}
That is to say, $\mathrm{\overline{Ric}}_{\overline M}\geq 0$. 
Applying Corollary 5.2 in \cite{Grigoryan} to $\partial_t u=\overline\Delta_{\overline M}u^\gamma$, we can derive that $u$ has finite propagation speed on $\overline M$, so that $u$ has finite propagation speed on $M$. 
\end{proof}

Recall that Li-Li \cite{Li-Li} proved the W-entropy-information formula for the Langevin deformation of flows between the gradient flow of the Boltzmann-Shannon entropy and the Benamou-Brenier geodesic on the $L^2$-Wasserstein space.  To extend their result to the Langevin deformation of flows for the R\'enyi entropy, we need the following variational formula. Indeed, the left hand side of $\eqref{W entrop-2}$, which is the linear combination of
 $\mathrm{Ent}_\gamma$ and its first and 
second variations,  as well as the square of the  norm of the 
gradient of $\mathrm{Ent}_\gamma$ (i.e., the Fisher information), is actually the left hand side of the $W$-entropy-information formula $\eqref{W entropy for L flow}$. 

\begin{theorem}\label{$W$-entropy}
Let $(M,g,\mu)$ be a weighted complete Riemannian manifold with bounded geometry condition, $f\in C^2(M)$ and $d\mu=e^{-f}dv$.  

\begin{itemize}

\item For any $c\in(0,+\infty)$, let $\alpha=\frac{u^\prime}{u}$, where $u$ is a solution to 
$\eqref{equation of u}$ and 
let $(\rho,\phi)$ be a smooth solution to the Langevin deformation $(\ref{Lang-gamma})$. Then 
\begin{equation}\label{W entrop-2}
\begin{aligned}
&\frac{d^2}{dt^2} \mathrm{Ent}_\gamma(\rho)+\left[2\alpha(t)\left(m(\gamma-1)+1\right)+\frac{1}{c^2}\right]\frac{d}{dt}\mathrm{Ent}_\gamma(\rho)\\
&\hskip1cm +\frac{1}{c^2}\|\nabla \mathrm{Ent}_\gamma(\rho)\|^2+\left((\gamma-1)m+(\gamma-1)^2m^2\right)\alpha^2(t) \mathrm{Ent}_\gamma(\rho)\\
=&\int_M \rho^\gamma\left[\left\|\nabla^2\phi-\alpha(t) g\right\|_{\mathrm{HS}}^2+\mathrm{Ric}_{m,n}(L)(\nabla\phi,\nabla\phi)\right.\\
&\left.\hskip1.5cm+(\gamma-1)(L\phi-m\alpha(t))^2+(m-n)\left(\frac{\nabla\phi\cdot\nabla f}{m-n}+\alpha(t)\right)^2\right]d\mu.
\end{aligned}
\end{equation}

\item For $c=0$, let $(\rho,\phi)$ be a smooth solution to the porous medium equation \eqref{PMEL}. 
%
Then
\begin{equation}\label{W entropy c=0}
\begin{aligned}
&\frac{d^2}{dt^2}\mathrm{Ent}_\gamma(\rho)+4\alpha(t)\left[m(\gamma-1)+1\right]\frac{d}{dt}\mathrm{Ent}_\gamma(\rho)+2\left((\gamma-1)m+(\gamma-1)^2m^2\right)\alpha^2(t)\mathrm{Ent}_\gamma(\rho)\\
=&2\int_M \rho^\gamma\left[\left\|\nabla^2\phi-\alpha(t) g\right\|_{\mathrm{HS}}^2+\mathrm{Ric}_{m,n}(L)(\nabla\phi,\nabla\phi)\right.\\
&\left.\hskip1.5cm+(\gamma-1)(L\phi-m\alpha(t))^2+(m-n)\left(\frac{\nabla\phi\cdot\nabla f}{m-n}+\alpha(t)\right)^2\right]d\mu.
\end{aligned}
\end{equation}

\item
For $c=+\infty$, let $(\rho, \phi)$ be a smooth solution to the geodesic flow $\eqref{Geodesic-BB}$. Then
\begin{equation}\label{W entropy c=infty}
\begin{aligned}
&\frac{d^2}{dt^2}\mathrm{Ent}_\gamma(\rho)+2\alpha(t)\left[m(\gamma-1)+1\right]\frac{d}{dt}\mathrm{Ent}_\gamma(\rho)+\left((\gamma-1)m+(\gamma-1)^2m^2\right)\alpha^2(t)\mathrm{Ent}_\gamma(\rho)\\
=&\int_M \rho^\gamma\left[\left\|\nabla^2\phi-\alpha(t) g\right\|_{\mathrm{HS}}^2+\mathrm{Ric}_{m,n}(L)(\nabla\phi,\nabla\phi)\right.\\
&\left.\hskip1.5cm+(\gamma-1)(L\phi-m\alpha(t))^2+(m-n)\left(\frac{\nabla\phi\cdot\nabla f}{m-n}+\alpha(t)\right)^2\right]d\mu.
\end{aligned}
\end{equation}
Moreover, if $\mathrm{Ric}_{m,n}(L)\geq 0$, then the left and the right hand sides of $(\ref{W entrop-2})$, $(\ref{W entropy c=0})$ and $(\ref{W entropy c=infty})$ are all non-negative. Moreover, the both sides vanish if and only if $(\rho(t), \phi(t))=(\rho_{c, m}(t), \phi_{c, m}(t))$ for $t>0$. 

\end{itemize}
\end{theorem}

\begin{proof}
Applying $(\ref{Hess of V})$ and $(\ref{ineqV})$ and calculating directly,  we have 
\begin{equation*}
\begin{aligned}
&\int_M \rho^\gamma\left[\|\nabla^2\phi\|_{\mathrm{HS}}^2+\mathrm{Ric}(L)(\nabla\phi,\nabla\phi)+(\gamma-1)(L\phi)^2\right]d\mu\\
=&\int_M \rho^\gamma\left[\left\|\nabla^2\phi-\alpha g\right\|_{\mathrm{HS}}^2-n\alpha^2+2\alpha\Delta\phi+\mathrm{Ric}_{m,n}(L)(\nabla\phi,\nabla\phi)+\frac{(\nabla\phi\cdot\nabla f)^2}{m-n}+(\gamma-1)(L\phi)^2\right]d\mu\\
=&\int_M \rho^\gamma\left[\left\|\nabla^2\phi-\alpha g\right\|_{\mathrm{HS}}^2-n\alpha^2+2\alpha(L\phi+\nabla\phi\cdot\nabla f)+\mathrm{Ric}_{m,n}(L)(\nabla\phi,\nabla\phi)\right.\\
&\quad\qquad \left.+\frac{(\nabla\phi\cdot\nabla f)^2}{m-n}+(\gamma-1)(L\phi)^2\right]d\mu\\	
=&\int_M \rho^\gamma\left[\left\|\nabla^2\phi-\alpha g\right\|_{\mathrm{HS}}^2+\mathrm{Ric}_{m,n}(L)(\nabla\phi,\nabla\phi)+(m-n)\left(\frac{\nabla\phi\cdot\nabla f}{m-n}+\alpha\right)^2\right.\\
&\qquad \quad \left.+(\gamma-1)(L\phi)^2+2\alpha L\phi-(m-n)\alpha^2-n\alpha^2\right]d\mu\\
=&\int_M \rho^\gamma\left[\left\|\nabla^2\phi-\alpha g\right\|_{\mathrm{HS}}^2+\mathrm{Ric}_{m,n}(L)(\nabla\phi,\nabla\phi)+(m-n)\left(\frac{\nabla\phi\cdot\nabla f}{m-n}+\alpha\right)^2\right.\\
&\qquad \quad \left.+(\gamma-1)(L\phi-m\alpha)^2+2\alpha\left(m(\gamma-1)+1\right)(L\phi)-m^2(\gamma-1)\alpha^2-m\alpha^2\right]d\mu\\
=&\int_M \rho^\gamma\left[\left\|\nabla^2\phi-\alpha g\right\|_{\mathrm{HS}}^2+\mathrm{Ric}_{m,n}(L)(\nabla\phi,\nabla\phi)+(m-n)\left(\frac{\nabla\phi\cdot\nabla f}{m-n}+\alpha\right)^2\right.\\
&\quad\left.+(\gamma-1)(L\phi-m\alpha)^2\right]d\mu-2\alpha\left(m(\gamma-1)+1\right)\frac{d}{dt}\mathcal{V}(\rho)-(\gamma-1)\left(m+(\gamma-1)m^2\right)\alpha^2\mathcal{V}(\rho).
\end{aligned}
\end{equation*}
Substituting to $(\ref{W entrop-1})$, $(\ref{2 varitaion for gra})$ and $(\ref{2 varitaion for geo})$, we obtain $(\ref{W entrop-2})$, $(\ref{W entropy c=0})$ and $(\ref{W entropy c=infty})$, respectively. 
\end{proof}

In particular, in the case $m=n$ and $L=\Delta$, we have 

\begin{corollary}  
Let $(M, g)$ be a complete Riemannian manifold with bounded geometry condition. For any $c\in(0,+\infty)$, let $\alpha(t)=\frac{u^{\prime}(t)}{u(t)}$, where $u$ is the solution to $(\ref{equation of u})$ and $(\rho,\phi)$ be a smooth solution to the Langevin deformation $(\ref{Lang-gamma})$ on $T\mathcal{P}_2(M, dv)$. Then
\begin{equation}\label{$W$-entropy 2'}
 \begin{aligned}
 &\frac{d^2}{dt^2}\mathrm{Ent}_\gamma(\rho)+\left[2\alpha(t)\left(n(\gamma-1)+1\right)+\frac{1}{c^2}\right]\frac{d}{dt}\mathrm{Ent}_\gamma(\rho)+\frac{1}{c^2}\|\nabla \mathrm{Ent}_\gamma(\rho)\|^2\\
 &\hskip2cm+\left((\gamma-1)n+(\gamma-1)^2n^2\right)\alpha^2(t)\mathcal{V}(\rho)\\
 =&\int_M \rho^\gamma\left[\left\|\nabla^2\phi-\alpha(t) g\right\|_{\mathrm{HS}}^2+\mathrm{Ric}(\nabla\phi,\nabla\phi)+(\gamma-1)(\Delta\phi-n\alpha(t))^2\right]dv.
 \end{aligned}
 \end{equation} 
For $c=0$,  let $(\rho,\phi)$ be a smooth solution to the porous medium equation \eqref{porous medium}. Then
\begin{equation}\label{W entropy c=0'}
\begin{aligned}
&\frac{d^2}{dt^2}\mathrm{Ent}_\gamma(\rho)+4\alpha(t)\left[n(\gamma-1)+1\right]\frac{d}{dt}\mathrm{Ent}_\gamma(\rho)+2\left((\gamma-1)n+(\gamma-1)^2n^2\right)\alpha^2(t)\mathrm{Ent}_\gamma(\rho)\\
=&2\int_M \rho^\gamma\left[\left\|\nabla^2\phi-\alpha(t) g\right\|_{\mathrm{HS}}^2+\mathrm{Ric}(\nabla\phi,\nabla\phi)+(\gamma-1)(\Delta\phi-n\alpha(t))^2\right]dv.
\end{aligned}
\end{equation}
For $c=+\infty$, let $(\rho, \phi)$ be a smooth solution to the geodesic flow $(\ref{Geodesic-BB})$ for $\mu=v$. Then
\begin{equation}\label{W entropy c=infty'}
\begin{aligned}
&\frac{d^2}{dt^2}\mathrm{Ent}_\gamma(\rho)+2\alpha(t)\left[n(\gamma-1)+1\right]\frac{d}{dt}\mathrm{Ent}_\gamma(\rho)+\left((\gamma-1)n+(\gamma-1)^2n^2\right)\alpha^2(t)\mathrm{Ent}_\gamma(\rho)\\
=&\int_M \rho^\gamma\left[\left\|\nabla^2\phi-\alpha(t) g\right\|_{\mathrm{HS}}^2+\mathrm{Ric}(\nabla\phi,\nabla\phi)+(\gamma-1)(\Delta\phi-n\alpha(t))^2\right]dv.
\end{aligned}
\end{equation}
Moreover, if $\mathrm{Ric}\geq 0$, then the left and the right hand sides of $(\ref{$W$-entropy 2'})$, $(\ref{W entropy c=0'})$ and $(\ref{W entropy c=infty'})$ are all non-negative. Moreover, the both sides vanish if and only if $(\rho(t), \phi(t))=(\rho_{c,n}(t), \phi_{c,n}(t))$ for $t>0$. 
\end{corollary}

More specially, for $M=\mathbb{R}^n$ and $m=n$, we have
\begin{theorem}   For any $c\in(0,+\infty)$, let $\alpha(t)=\frac{u^{\prime}(t)}{u(t)}$, where $u$ is the solution to $(\ref{equation of u})$ and $(\rho,\phi)$ be a smooth solution to the Langevin deformation $(\ref{Lang-gamma})$ on $T\mathcal{P}_2(\mathbb R^n,dx)$. Then 
  \begin{equation}\label{$W$-entropy 2c}
 \begin{aligned}
 &\frac{d^2}{dt^2}\mathrm{Ent}_\gamma(\rho)+\left[2\alpha(t)\left((\gamma-1)n+1\right)+\frac{1}{c^2}\right]\frac{d}{dt}\mathrm{Ent}_\gamma(\rho)+\frac{1}{c^2}\|\nabla \mathrm{Ent}_\gamma(\rho)\|^2\\
 &+\left((\gamma-1)n+(\gamma-1)^2n^2\right)\alpha^2(t)\mathcal{V}(\rho)\\
 =&\int_{\mathbb R^n} \rho^\gamma\left[\left\|\nabla^2\phi-\alpha(t) g\right\|_{\rm HS}^2+(\gamma-1)(\Delta\phi-n\alpha(t))^2\right]dx.
 \end{aligned}
 \end{equation} 
For $c=0$, it holds 
\begin{equation}\label{W entropy 2c=0}
\begin{aligned}
&\frac{d^2}{dt^2}\mathrm{Ent}_\gamma(\rho)+4\alpha(t)\left[n(\gamma-1)+1\right]\frac{d}{dt}\mathrm{Ent}_\gamma(\rho)+2\left((\gamma-1)n+(\gamma-1)^2n^2\right)\alpha^2(t)\mathrm{Ent}_\gamma(\rho)\\
=&2\int_{\mathbb R^n} \rho^\gamma\left[\left\|\nabla^2\phi-\alpha(t) g\right\|_{\rm HS}^2+(\gamma-1)(\Delta\phi-n\alpha(t))^2\right]dx.
\end{aligned}
\end{equation}
For $c=+\infty$, it holds 
\begin{equation}\label{W entropy 2c=infty}
\begin{aligned}
&\frac{d^2}{dt^2}\mathrm{Ent}_\gamma(\rho)+2\alpha(t)\left[n(\gamma-1)+1\right]\frac{d}{dt}\mathrm{Ent}_\gamma(\rho)+\left((\gamma-1)n+(\gamma-1)^2n^2\right)\alpha^2(t)\mathrm{Ent}_\gamma(\rho)\\
=&\int_{\mathbb R^n} \rho^\gamma\left[\left\|\nabla^2\phi-\alpha(t) g\right\|_{\rm HS}^2+(\gamma-1)(\Delta\phi-n\alpha(t))^2\right]dx.
\end{aligned}
\end{equation}
In particular, the left and the right hand sides of $(\ref{$W$-entropy 2c})$, $(\ref{W entropy 2c=0})$ and $(\ref{W entropy 2c=infty})$ are all non-negative. Moreover, the both sides vanish if and only if 
$(\rho(t), \phi(t))=(\rho_{c, n}(t), \phi_{c, n}(t))$ for $t>0$.  
\end{theorem}

\section{$W$-entropy-information formula and rigidity theorem}

In this section, we give the explicit construction of  the spacial solution $(\rho_{c, m}, \phi_{c, m})$ 
to the Langevin deformation of flows $(\ref{Lang-gamma})$ on $\mathcal{P}_2(\mathbb{R}^m, dx)$. We call it the reference model.  
 We use the variational formulae of the R\'enyi entropy 
in the previous section to derive the $W$-entropy-information formulae for weighted porous medium equation, 
geodesic flow and the Langevin deformation of flows between the gradient flow for the R\'enyi entropy 
and the geodesic flow, respectively. Moreover, we prove
 that $(\rho_{c, m}, \phi_{c, m})$ is the rigidity model 
for the $W$-entropy-information formula on complete Riemannian manifolds with bounded geometry condition and CD$(0, m)$-condition.

\subsection{Reference model}\label{ref model sec}

Now we introduce the reference model. Let $m\in \mathbb{N}$ and $(M,g,d\mu)=(\mathbb R^m,g_{\rm E},dx)$, where $g_{\rm E}$ is the Euclidean metric on $\mathbb R^m$ and $dx$ is the standard Lebesgue volume measure on $\mathbb R^m$. 

\begin{theorem}\label{ref model thm} 
Assuming $\rho_0$ is defined by  
\begin{equation*}\label{rho_0}
 	\frac{\gamma}{\gamma-1}\rho_0^{\gamma-1}(y)=\max\left\{\lambda-\frac{k}{2}|y|^2,0\right\}, \quad k=\frac{1}{m(\gamma-1)+2},
\end{equation*} 
where $\lambda$ is a constant such that $\int_{\mathbb{R}^m} \rho_0(y)dy=1$. Fix $c\in(0,+\infty)$ and $T>0$. Let $u(t)$ be a solution to ODE \eqref{equation of u} on $[\delta,T]\subset(0,+\infty)$. Let $\alpha_1=\frac{u^\prime}{u}$ and $\beta_1$ be a solution to the following equation
\begin{equation}\label{beta equation}
	c^2\beta_1^{\prime}(t)+\beta_1(t)+\frac{\lambda}{u^{(\gamma-1)m}(t)}=0. 
\end{equation}
Let $\alpha_2$ be a solution to 
\begin{equation*}
	c^2\left(\alpha_2^\prime+\alpha_2^2\right)+\alpha_2=0.
\end{equation*}
Define 
\begin{eqnarray}
	\rho_{c,m}(x,t)&=&\frac{1}{u^m(t)}\rho_0\left(\frac{x}{u(t)}\right),\label{rho_c,m}\\
	\phi_{c,m}(x,t)&=&\left\{
	\begin{aligned}
	\frac{\alpha_1(t)}{2}|x|^2+\beta_1(t),\quad \frac{x}{u(t)}\in \mathrm{supp}~\rho_0,\\
	\frac{\alpha_2(t)}{2}|x|^2+e^{-c^2t},\quad \frac{x}{u(t)}\notin \mathrm{supp}~\rho_0.
	\end{aligned}
	\right.
\end{eqnarray} 
Then $(\rho_{c,m},\phi_{c,m})$ is a special solution to the Langevin deformation of flows $(\ref{Lang-gamma})$ on $T\mathcal{P}_2^{\infty}(\mathbb R^m, dx)$.
\end{theorem}
\begin{proof} 
These special solutions can be verified by directly calculus. Applying the same strategy as in \cite{Li-Li}, we can construct the special solution by the typical method of characteristics to solve the continuity equation. To save the length of this paper, we will omit the details. 
\end{proof}

\begin{remark}
In particular, we have 
\begin{itemize}
\item When $c=0$, the Langevin deformation of flows becomes gradient flow of the R\'enyi entropy $\mathrm{Ent}_{\gamma}$ on $\mathcal{P}_2^{\infty}(\mathbb R^m,dx)$, i.e, the porous medium equation \eqref{porous medium}. 
Let $u$ be a positive solution to 
\begin{equation*}
	kuu^{\prime\prime}+(1-k)(u^{\prime})^2=0
\end{equation*}
Then the following $\rho$ is a special solution to $\eqref{porous medium}$ 
\begin{eqnarray*}
    \rho(x,t)=\frac{1}{u^{m}}\rho_0\left(\frac{x}{u}\right), 
 \end{eqnarray*}
which is the Barenblatt's self-similar solution. In particular, we can take $u=t^{k}$ and denote the special solution by $\rho_{0,m}$. That is 
\begin{equation}\label{rho_0,m}
 	\rho_{0,m}(x,t)=\frac{1}{t^{km}}\rho_0\left(\frac{x}{t^k}\right). 
 \end{equation} 
Indeed, for any $\delta>0$, we can also define the $W$-entropy for porous medium equation by
$$W_{0,m}(\rho,t)=\frac{d}{dt}\left((t+\delta)^{2k-2}H_{0,m}(\rho(t))\right), \quad t\geq 0.$$ 

\item When $c=\infty$, the Langevin deformation of flows become the Benamou-Brenier geodesic flow $(\ref{Geodesic-BB})$ on $T\mathcal{P}_2^{\infty}(\mathbb R^m,dx)$. Let $u$ be a solution to
\begin{equation*}
	u^{\prime\prime}=0,
\end{equation*}
and let $\alpha=\frac{u^\prime}{u}$, then the following $(\rho,\phi)$ is a special solution to 
$\eqref{Geodesic-BB}$ on $T\mathcal{P}_2^{\infty}(\mathbb R^m,dx)$, 
\begin{eqnarray*}
    \rho(x,t)&=&\frac{1}{u(t)^{m}}\rho_0\left(\frac{x}{u(t)}\right),\\
    \phi(x,t)&=&\frac{\alpha(t)}{2}|x|^2.\label{phi_infty,m}
\end{eqnarray*}
In particular, we can take $u=t$, $\alpha=\frac{1}{t}$ and denote the special solution by $(\rho_{\infty,m},\phi_{\infty,m})$. That is 
\begin{eqnarray}
    \rho_{\infty,m}(x,t)&=&\frac{1}{t^{m}}\rho_0\left(\frac{x}{t}\right),\label{rho_infty,m}\\
    \phi_{\infty,m}(x,t)&=&\frac{1}{2t}|x|^2.\label{phi_infty,m}
\end{eqnarray}
Similarly, for any $\delta>0$, we can also define the $W$-entropy for the geodesic flow by 
$$W_{\infty,m}(\rho,t)=\frac{d}{dt}\left((t+\delta)^{\frac{1}{k}-1}H_{\infty,m}(\rho(t))\right), \quad t\geq 0.$$
\end{itemize}
\end{remark}

\begin{remark}As pointed out by Li-Li \cite{Li-Li}, Equation $(\ref{equation of u})$ can be regarded as the Langevin deformation of flows on 
$T\hspace{0.2mm}\mathbb{R}^*=\mathbb{R}^*\times \mathbb{R}$ with the potential $V(x)=-\frac{k^2}{2k-1}x^{2-\frac{1}{k}}$ on
$\mathbb{R}^*:=\mathbb{R} \setminus\{0\}$. Indeed, the Langevin deformation equation with a general external potential $V$ 
on $T\mathbb{R}^*$ is given by
\begin{eqnarray*}
\dot x&=&v,\\
c^2\dot v&=&-v-V'(x).
\end{eqnarray*}
The second equation is the Newton equation for a particle with mass $c^2$ moving in a fluid with external potential $V$ and damping force $-v$. This yields
\begin{equation*}
	c^2\ddot{x}+\dot{x}=kx^{1-\frac{1}{k}},
\end{equation*}
which is exactly the equation $(\ref{equation of u})$ of $u$.
\end{remark}
\subsection{Proof of $W$-entropy(-information) formulae}

Now we prove the $W$-entropy(-information) formulae and rigidity theorems. 
%
%
\begin{proof}[Proof of Theorem \ref{W entropy for gradient flow}]
From Theorem $\ref{ref model thm}$, we can take $\alpha(t)=\frac{k}{t}$ for $c=0$. Notice that $m(\gamma-1)=\frac{1}{k}-2$, so we rewrite $(\ref{W entropy c=0})$ as
\begin{equation}\label{W entrop-3}
\begin{aligned}
&\frac{d^2}{dt^2}\mathrm{Ent}_\gamma(\rho(t))+\frac{4}{t}(1-k)\frac{d}{dt}\mathrm{Ent}_\gamma(\rho(t))+\frac{2(1-2k)(1-k)}{t^2}\mathrm{Ent}_\gamma(\rho(t))\\
=&2\int_M \rho^\gamma\left[\left\|\nabla^2\phi-\frac{kg}{t}\right\|_{\mathrm{HS}}^2+\mathrm{Ric}_{m,n}(L)(\nabla\phi,\nabla\phi)\right.\\
&\left.+(\gamma-1)\left(L\phi-\frac{mk}{t}\right)^2+(m-n)\left(\frac{\nabla\phi\cdot\nabla f}{m-n}+\frac{k}{t}\right)^2\right]d\mu.
\end{aligned}
\end{equation}
For $c=0$, we have $\eta(t)=\frac{1}{t^k}$ and $\rho_{0,m}(x,t)=\frac{1}{t^{mk}}\rho_0\left(\frac{x}{t^k}\right)$, so
\begin{equation*}
\begin{aligned}
	\mathrm{Ent}_\gamma(\rho_{0,m})&=\frac{1}{\gamma-1}\int_{\mathbb R^m} \rho_{0,m}^\gamma(x,t) dx= \frac{1}{\gamma-1}\int_{\mathbb R^m} \frac{1}{t^{mk\gamma}}\rho_0^\gamma(y)d(t^ky)\\
	&=\frac{1}{(\gamma-1)t^{mk(\gamma-1)}}\int_{\mathbb R^m} \rho_0^\gamma(y)dy=At^{-mk(\gamma-1)}=At^{2k-1},
\end{aligned}	
\end{equation*}	
where 
$$A=\mathrm{Ent}_\gamma(\rho_0)=\frac{1}{\gamma-1}\int_{\mathbb R^m} \rho_0^\gamma(y)dy$$ is a constant depending only on $\gamma$ and $\rho_0$. Therefore, 
\begin{equation*}
	\frac{d}{dt}\mathrm{Ent}_\gamma(\rho_{0,m})=A\left(2k-1\right)t^{2k-2},
\end{equation*}
\begin{equation*}
	\frac{d^2}{dt^2}\mathrm{Ent}_\gamma(\rho_{0,m})=A\left(2k-1\right)\left(2k-2\right)t^{2k-3}.
\end{equation*}
Then it's easy to check
\begin{equation*}
\frac{d^2}{dt^2}\left(t^{2-2k}\mathrm{Ent}_{\gamma}(\rho_{0,m}(t))\right)=0.
\end{equation*}
Thus we have 
\begin{equation*}
\begin{aligned}
&\frac{d}{dt}W_{0,m}(\rho,t)=\frac{d^2}{dt^2}\left(t^{2-2k}\mathrm{Ent}_\gamma(\rho(t))\right)\\
=&t^{2-2k}\left(\frac{d^2}{dt^2}\mathrm{Ent}_\gamma(\rho(t))+\frac{4}{t}(1-k)\frac{d}{dt}\mathrm{Ent}_\gamma(\rho(t))+\frac{2(1-2k)(1-k)}{t^2}\mathrm{Ent}_\gamma(\rho(t))\right).
\end{aligned}
\end{equation*}
Then we can obtain $(\ref{W entropy eq for gra})$ by applying $(\ref{W entrop-3})$. 

The rigidity part can be proved by the same argument as used in \cite{Fang-Li-Zhang, Li2012MA, Li-Li}. It is based on the fact that if there exists a function $\phi$ on complete Riemannian manifold $M$ satisfying $\nabla^2 \phi = \lambda g$ for some $\lambda\in\mathbb R\setminus\{0\}$, then $(M, g)$ must be isometric to Euclidean space $\mathbb R^n$. For this, see  \cite{Ni, Petersen-Wylie, Pigola}. 

\end{proof}
%
%
\begin{proof}[Proof of Theorem \ref{W entropy for geodesic}]
From Theorem $\ref{ref model thm}$, we can take $\alpha(t)=\frac{1}{t}$ for $c=\infty$, so we rewrite $(\ref{W entropy c=infty})$ by 
\begin{equation}\label{W entrop-4}
\begin{aligned}
	&\frac{d^2}{dt^2}\mathrm{Ent}_\gamma(\rho(t))+\frac{2(1-k)}{tk}\frac{d}{dt}\mathrm{Ent}_\gamma(\rho(t))+\frac{(1-2k)(1-k)}{t^2k^2}\mathrm{Ent}_\gamma(\rho(t))\\
	&=\int_M \rho^\gamma\left[\left\|\nabla^2\phi-\frac{g}{t}\right\|_{\mathrm{HS}}^2+\mathrm{Ric}_{m,n}(L)(\nabla\phi,\nabla\phi)\right.\\
	&\left.+(\gamma-1)\left(L\phi-\frac{m}{t}\right)^2+(m-n)\left(\frac{\nabla\phi\cdot\nabla f}{m-n}+\frac{1}{t}\right)^2\right]d\mu.
\end{aligned}
\end{equation}
For $c=\infty$, we have $\eta(t)=\frac{1}{t}$ and $\rho_{\infty,m}(x,t)=\frac{1}{t^m}\rho_0\left(\frac{x}{t}\right)$, so
\begin{equation*}
\begin{aligned}
	\mathrm{Ent}_\gamma(\rho_{\infty,m})&=\frac{1}{\gamma-1}\int_{\mathbb R^m} \rho_{\infty,m}^\gamma(x,t) dx= \frac{1}{\gamma-1}\int_{\mathbb R^m} \frac{1}{t^{m\gamma}}\rho_0^\gamma(y)d(ty)\\
	&=\frac{1}{(\gamma-1)t^{m(\gamma-1)}}\int_{\mathbb R^m} \rho_0^\gamma(y)dy=At^{-m(\gamma-1)}=At^{2-\frac{1}{k}},
\end{aligned}	
\end{equation*}	
where $A=\mathrm{Ent}_\gamma(\rho_0)$ is a constant depending only on $\gamma$ and $\rho_0$. Therefore, 
\begin{equation*}
	\frac{d}{dt}\mathrm{Ent}_\gamma(\rho_{\infty,m})=A\left(2-\frac{1}{k}\right)t^{1-\frac{1}{k}},
\end{equation*}
\begin{equation*}
	\frac{d^2}{dt^2}\mathrm{Ent}_\gamma(\rho_{\infty,m})=A\left(2-\frac{1}{k}\right)\left(1-\frac{1}{k}\right)t^{-\frac{1}{k}}.
\end{equation*}
Similarly, it's easy to check
\begin{equation*}
\frac{d^2}{dt^2}\left(t^{\frac{1}{k}-1}\mathrm{Ent}_{\gamma}(\rho_{\infty,m}(t))\right)=0.
\end{equation*}
Thus we have 
\begin{equation*}
\begin{aligned}
&\frac{d}{dt}W_{\infty,m}(\rho,t)=\frac{d^2}{dt^2}\left(t^{\frac{1}{k}-1}\mathrm{Ent}_\gamma(\rho(t))\right)\\
=&t^{\frac{1}{k}-1}\left(\frac{d^2}{dt^2}\mathrm{Ent}_\gamma(\rho(t))+\frac{2(1-k)}{tk}\frac{d}{dt}\mathrm{Ent}_\gamma(\rho(t))+\frac{(1-2k)(1-k)}{t^2k^2}\mathrm{Ent}_\gamma(\rho(t))\right).
\end{aligned}
\end{equation*}
Then we can obtain $(\ref{W entropy for geodesic})$ by applying $(\ref{W entrop-4})$. The rigidity part can be proved by similar argument as used in \cite{Fang-Li-Zhang, Li2012MA, Li-Li} based on \cite{Ni, Petersen-Wylie, Pigola}. 

\end{proof}

Finally, we consider the case of $c\in(0,+\infty)$ and prove the $W$-entropy-information formula (i.e. Theorem \ref{W entropy of Langevin flow}) for the Langevin deformation of flow $\eqref{Lang-gamma}$.

\begin{proof}[Proof of Theorem \ref{W entropy of Langevin flow}]
For the rigidity model given in Theorem $\ref{ref model thm}$, we have  
\begin{equation*}
	\mathrm{Ent}_\gamma(\rho_{c,m})=Au^{2-\frac{1}{k}}, \quad A=\mathrm{Ent}_\gamma(\rho_0).
\end{equation*}	
Since $u^\prime=\alpha u$, we have 
\begin{equation*}
	\frac{d}{dt} \mathrm{Ent}_\gamma(\rho_{c,m})=A\left(2-\frac{1}{k}\right)u^{2-\frac{1}{k}}\alpha,
\end{equation*}
\begin{equation*}
	\frac{d^2}{dt^2} \mathrm{Ent}_\gamma(\rho_{c,m})=A\left(2-\frac{1}{k}\right)u^{2-\frac{1}{k}}\left(\left(2-\frac{1}{k}\right)\alpha^2+\alpha^\prime\right).
\end{equation*}
By the definition of $W_{c,m}(\rho,t)$, we have 
\begin{equation}\label{dW_c,m}
\begin{aligned}
	\frac{d}{dt}W_{c,m}(\rho,t)=b(t)&\left[\left(\frac{d^2}{dt^2} \mathrm{Ent}_\gamma(\rho(t))+\frac{a+b^\prime}{b}\frac{d}{dt} \mathrm{Ent}_\gamma(\rho(t))+\frac{a^\prime}{b}\mathrm{Ent}_\gamma(\rho(t))\right)\right.\\
	&\left.-\left(\frac{a^\prime}{b}\mathrm{Ent}_\gamma(\rho_{c,m})+\frac{a+b^\prime}{b}\frac{d}{dt} \mathrm{Ent}_\gamma(\rho_{c,m})+\frac{d^2}{dt^2} \mathrm{Ent}_\gamma(\rho_{c,m})\right)\right],
\end{aligned}
\end{equation}
where
\begin{eqnarray*}
&&\frac{a^\prime}{b}\mathrm{Ent}_\gamma(\rho_{c,m})+{a+b^\prime\over b}\frac{d}{dt}\mathrm{Ent}_\gamma(\rho_{c,m})+\frac{d^2}{dt^2}\mathrm{Ent}_\gamma(\rho_{c,m})\\
&=&\left(\frac{1}{k}-2\right)\left(\frac{1}{k}-1\right)\alpha^2Au^{2-\frac{1}{k}}+\left[2\alpha\left(\frac{1}{k}-1\right)+\frac{1}{c^2}\right]A\left(2-\frac{1}{k}\right)u^{2-\frac{1}{k}}\alpha\\
&&+A\left(2-\frac{1}{k}\right)u^{2-\frac{1}{k}}\left[\left(2-\frac{1}{k}\right)\alpha^2+\alpha^\prime\right]\\
&=&\left(\frac{1}{k}-2\right)Au^{2-\frac{1}{k}}\left[\left(\frac{1}{k}-1\right)\alpha^2-2\alpha^2\left(\frac{1}{k}-1\right)-\frac{\alpha}{c^2}-\left(2-\frac{1}{k}\right)\alpha^2-\alpha^\prime\right]\\
&=&\left(\frac{1}{k}-2\right)Au^{2-\frac{1}{k}}\left[-\alpha^2-\alpha^\prime-\frac{\alpha}{c^2}\right].
\end{eqnarray*}
Since $\alpha=\frac{u^\prime}{u}$ and $c^2u^{\prime\prime}+u^{\prime}=ku^{1-\frac{1}{k}}$, we have 
\begin{equation}\label{3parts}
\frac{a^\prime}{b}\mathrm{Ent}_{\gamma}(\rho_{c,m})+\frac{a+b^\prime}{b}\frac{d}{dt}\mathrm{Ent}_{\gamma}(\rho_{c,m})+\frac{d^2}{dt^2}\mathrm{Ent}_{\gamma}(\rho_{c,m})=(2k-1)\frac{A}{c^2}u^{2\left(1-\frac{1}{k}\right)}	.
\end{equation}
Denote
\begin{equation*}
	B(|x|)=\max\left\{\left(\tilde\lambda-\frac{\gamma-1}{2\gamma}|x|^2\right)^{\frac{1}{\gamma-1}},0\right\},
\end{equation*}
where $\tilde\lambda$ is the normalized constant such that $\int_{\mathbb{R}^m} B(|x|)dx=1$. Denote the R\'enyi entropy of $B$ by 
\begin{equation*}
	\mathrm{Ent}_\gamma(B)=\frac{1}{\gamma-1}\int_{\mathbb R^m}B(|x|)^\gamma dx,
\end{equation*}
and the second moment of $B$ by 
\begin{equation*}
	M_2(B)=\int_{\mathbb R^m}\frac{|x|^2}{2}B(|x|) dx.
\end{equation*}
According to Toscani \cite{Toscani05}, it holds
\begin{equation*}
	\frac{\mathrm{Ent}_\gamma(B)}{M_{2}(B)}=\frac{2k}{1-2k}.
\end{equation*}
Since 
\begin{equation*}
	\rho_0(y)=\max\left\{\left(\lambda-\frac{k(\gamma-1)}{2\gamma}|y|^2\right)^{\frac{1}{\gamma-1}},0\right\},
\end{equation*}
we can prove the following equality by scaling
\begin{equation*}
	\frac{\mathrm{Ent}_\gamma(\rho_0)}{M_{2}(\rho_0)}=\frac{k^2}{1-2k}.
\end{equation*}
By direct calculus, we have
\begin{eqnarray*}
\|\nabla \mathrm{Ent}_\gamma(\rho_{c,m})\|^2&=&\int_{\mathbb R^m}\left|\nabla \left(\frac{\gamma}{\gamma-1}\rho_{c,m}^{\gamma-1}(x)\right)\right|^2\rho_{c,m}(x)dx\\
&=&\int_{\mathbb R^m} k^2u^{-\frac{2}{k}}|x|^2\frac{1}{u^m}\rho_0\left(\frac{x}{u}\right)dx\\
&=&k^2u^{2\left(1-\frac{1}{k}\right)}\int_{\mathbb R^m}|y|^2\rho_0(y)dy\\
&=&(1-2k)u^{2\left(1-\frac{1}{k}\right)}\mathrm{Ent}_\gamma(\rho_0)\\
&=&(1-2k)u^{2\left(1-\frac{1}{k}\right)}A.
\end{eqnarray*}
Combining with $(\ref{3parts})$, we have
\begin{equation*}
\frac{a^\prime}{b}\mathrm{Ent}_\gamma(\rho_{c,m})+\frac{a+b^\prime}{b}\frac{d}{dt}\mathrm{Ent}_\gamma(\rho_{c,m})+\frac{d^2}{dt^2}\mathrm{Ent}_\gamma(\rho_{c,m})=-\frac{1}{c^2}\|\nabla \mathrm{Ent}_\gamma(\rho_{c,m})\|^2.	
\end{equation*}
Combining with $({\ref{eq of ab}})$, $(\ref{dW_c,m})$ and $(\ref{W entrop-2})$, the $W$-entropy formula $(\ref{W entropy for L flow})$ is proved. 

Next we prove the existence and uniqueness of solution to the equations of $a$ and $b$.
\begin{equation*}
	a=\left[2\alpha\left(\frac{1}{k}-1\right)+\frac{1}{c^2}\right]b-b^\prime.
\end{equation*}
Taking  derivative on the both sides, we have 
\begin{equation*}
	a^\prime=2\alpha^\prime\left(\frac{1}{k}-1\right)b+\left[2\alpha\left(\frac{1}{k}-1\right)+\frac{1}{c^2}\right]b^\prime-b^{\prime\prime}.
\end{equation*}
That is
\begin{equation*}
	b^{\prime\prime}=2\alpha^\prime\left(\frac{1}{k}-1\right)b+\left[2\alpha\left(\frac{1}{k}-1\right)+\frac{1}{c^2}\right]b^\prime-a^\prime.
\end{equation*}
On the other hand
\begin{equation*}
	a^\prime=\left(\frac{1}{k}-2\right)\left(\frac{1}{k}-1\right)\alpha^2b.
\end{equation*}
Thus 
\begin{equation}\label{2nd eq of b}
	b^{\prime\prime}-\left[\frac{2u^\prime}{u}\left(\frac{1}{k}-1\right)+\frac{1}{c^2}\right]b^\prime-\left(\frac{1}{k}-1\right)\left(\frac{2u^{\prime\prime}}{u}-\frac{\left(u^{\prime}\right)^2}{ku^2}\right)b=0,
\end{equation}
which is a linear second order ODE. Since there exists a unique smooth solution $u>0$ to the equation $(\ref{equation of u})$ on $(0,T]$ when $T$ is small enough, we can obtain the existence and uniqueness of solution to $(\ref{2nd eq of b})$ with initial data $b(0)$ and $b^\prime(0)$. Thus the existence and uniqueness of $a$ can also be obtained. 

The rigidity part can be proved by similar argument as use in Theorem $\ref{W entropy for geodesic}$. 
\end{proof}

\subsection{Two corollaries of the $W$-entropy formula}\label{section of cor}

As applications of Theorem $\ref{W entropy for gradient flow}$ and Theorem $\ref{W entropy for geodesic}$, 
we have the following corollaries on the R\'enyi entropy along the gradient flow and geodesic flow. 


For the case $c=0$, we have the following 
\begin{corollary}
If $\mathrm{Ric}_{m,n}(L)\geq 0$, then the quantity $t^{2-2k}\left(\mathrm{Ent}_\gamma(\rho(t))-\mathrm{Ent}_\gamma(\rho_{0,m}(t))\right)$ is convex along any smooth solution to the porous medium equation \eqref{PMEL} on $\mathcal{P}_2^\infty(M,\mu)$. \end{corollary}

\begin{remark}\label{remLL0}
In \cite{Li-Li}, Li and Li proved that  
the quantity $t\left(\mathrm{Ent}(\rho(t))-\mathrm{Ent}(\rho_{0, m}(t))\right)$ is convex in $t$ 
for any positive solution to the heat equation $\partial_t \rho=L\rho$ on a Riemannian manifold $M$ 
with CD$(0, m)$-condition. Here $\rho_{0, m}(x, t)={1\over (4\pi t)^{m/2}}e^{-{\|x\|^2\over 4t}}$ is the fundamental solution to the heat  equation $\partial_t \rho=\Delta \rho$ on $\mathbb{R}^m$ when $m\in \mathbb{N}$. 
Note that
$\mathrm{Ent}(\rho_{0, m}(t))=-{m\over 2}\log (4\pi et)$. Thus, the quantity 
$t\mathrm{Ent}_{\gamma}(\rho(t))+{m\over 2}t\log t$ is convex in $t$ along the heat equation $\partial_t \rho=L\rho$ 
on a Riemannian manifold $M$ with CD$(0,m)$-condition.

\end{remark}

For any $t\in[\delta,T]\subset(0,+\infty)$, there exists some $s\in[\delta,t]$ such that 
\begin{equation*}
t^{2-2k}H_{0,m}(t)-\delta^{2-2k}H_{0,m}(\delta)=W_{0,m}(\delta)(t-\delta)+\left.\frac{d}{dt}\right|_{t=s}W_{0,m}(t)(t-\delta)^2. 
\end{equation*}
Assume that the initial value of $\eqref{PMEL}$ satisfies 
$$H_{0,m}(\delta)=\mathrm{Ent}_{\gamma}(\rho(\delta))-\mathrm{Ent}_{\gamma}(\rho_{0,m}(\delta)\geq 0.$$ 
Moreover, if we assume the Fisher information 
$$I(\rho(t))=\frac{d}{dt}\mathrm{Ent}_{\gamma}(\rho(t))=-\int_M \left|\nabla\left(\frac{\gamma}{\gamma-1}\rho^{\gamma-1}\right)\right|^2\rho d\mu$$
satisfies 
$$I(\rho(\delta))\geq I(\rho_{0,m}(\delta)),$$
then we can derive 
$$W_{0,m}(\delta)=\left.\frac{d}{dt}\right|_{t=\delta}\left(t^{2-2k}H_{0,m}(t)\right)=(2-2k)\delta^{1-2k}H_{0,m}(\delta)+\delta^{2-2k}(I(\rho(\delta))-I(\rho_{0,m}(\delta)))\geq 0.$$
Thus, on $[\delta, T]$, we have 
\begin{equation*}
t^{2-2k}H_{0,m}(t)= t^{2-2k}\left(\mathrm{Ent}_{\gamma}(\rho(t))-\mathrm{Ent}_{\gamma}(\rho_{0,m}(t))\right)=\left.\frac{d}{dt}\right|_{t=s}W_{0,m}(t)(t-\delta)^2\geq0. 
\end{equation*}
Hence we obtain 

\begin{corollary}\label{corc0}
Let $(M, g, \mu)$ be a weighted compact or a complete Riemannian manifold with 
bounded geometry condition. Let $\rho_{0,m}$ be the special solution to $\eqref{PMEL}$ defined as $\eqref{rho_0,m}$ on $[\delta,T]\subset(0,+\infty)$. Let $\rho$ be a solution to $\eqref{PMEL}$ on 
$(M, g, \mu)$. Assume that $\mathrm{Ent}_{\gamma}(\rho(\delta)\geq \mathrm{Ent}_{\gamma}(\rho_{0,m}(\delta))$ and $I(\rho(\delta))\geq I(\rho_{0,m}(\delta))$. If $\mathrm{Ric}_{m,n}(L)\geq 0$, 
then we have the following comparison formula 
\begin{equation*}
\mathrm{Ent}_{\gamma}(\rho(t))\geq \mathrm{Ent}_{\gamma}(\rho_{0,m}(t)), \quad \forall t\in[\delta,T].  
\end{equation*}
\end{corollary}

\begin{remark} In \cite{LNVV}, Lu-Ni-Vazquez-Villani introduced 
the $W$-entropy for the porous medium equation $\partial_t u=\Delta u^\gamma$ by 
\begin{equation*}
		W(\rho,t)=\frac{d}{dt}\left(t^{2-2k}\mathrm{Ent}_\gamma(\rho(t))\right).
	\end{equation*}
	They proved the $W$-entropy formula and the monotonicity of the $W$-entropy on complete 
	Riemannian manifold with non-negative Ricci curvature. 	
See also Huang-Li \cite{Huang-Li}  for an extension of Lu-Ni-Vazquez-Villani's result to the porous medium equation $\partial_t u=Lu^\gamma$ on compact Riemannian manifolds with weighted volume measure. Comparing with their results, the $W$-entropy we defined is indeed the relative $W$-entropy and we can check that
\begin{equation*}
\frac{d^2}{dt^2}\left(t^{2-2k}\mathrm{Ent}_\gamma(\rho_{0,m}(t))\right)=0.
\end{equation*}

\end{remark}

For the case $c=+\infty$, we have the following

\begin{corollary}
If $\mathrm{Ric}_{m,n}(L)\geq 0$, then the quantity $t^{\frac{1}{k}-1}\left(\mathrm{Ent}_\gamma(\rho(t))-\mathrm{Ent}_\gamma(\rho_{0,m}(t))\right)$ is convex along the 
Benamou-Brenier geodesic flow on $\mathcal{P}_2^\infty(M,\mu)$. 
\end{corollary}

\begin{remark} Let  $\rho_{\infty, m}(x, t)={1\over (4\pi t^2)^{m/2}}e^{-{\|x\|^2\over 4t^2}}$ and $\phi_{\infty, m}(x, t)={\|x\|^2\over 2t}$. 
In \cite{Li-Li}, Li and Li proved that  $(\rho_{\infty, m}, \phi_{\infty, m})$ is a special solution of the Benamou-Brenier geodesic on the $L^2$-Wasserstein space $\mathcal{P}_{2}(\mathbb{R}^m, dx)$, and the quantity $t\left(\mathrm{Ent}(\rho(t))-\mathrm{Ent}(\rho_{\infty, m}(t))\right)$ is convex in $t$ along the Benamou-Brenier geodesic on the $L^2$-Wasserstein space $\mathcal{P}_{2}(M, \mu)$ over a Riemannian manifold $M$ with CD$(0, m)$-condition. Note that $\mathrm{Ent}(\rho_{\infty, m}(t))=-{m\over 2}\log (4\pi et^2)$. This recaptures the convexity of the quantity $t\mathrm{Ent}_{\gamma}(\rho(t))+mt\log t$ along the Benamou-Brenier geodesic in $\mathcal{P}_{2}(M,d\mu)$ under the CD$(0,m)$-condition, which was first proved by Lott \cite{Lott09}. In \cite{Lott09}, Lott proved the convexity of the quantity $t \mathcal{U}_{\nu}(\rho(t))+mt\log t$ along the Benamou-Brenier geodesic flow on $\mathcal{P}_{2}(M,d\mu)$ under the CD$(0,m)$-condition, where $\mathcal{U}_\nu$ is in the class of functions $DC_\infty$. In \cite{Li2024}, the third named author proved that the convexity of the quantity $tH_p(\rho(t))+mt\log t$ along the Benamou-Brenier geodesic in $\mathcal{P}_{2}(M,d\mu)$ under CD$(0,m)$-condition, where
\begin{equation*}
H_p(\rho):=\frac{1}{p-1}\log\int_M\rho^p d\mu,\quad p\geq 1.
\end{equation*} 
The rigidity theorem is also proved in \cite{Li2024}.

\end{remark}

Similarly to the proof of Corollary \ref{corc0}, we have the following 
\begin{corollary}
Let $(M, g, \mu)$ be a weighted compact or a complete Riemannian manifold with bounded geometry condition. Let $(\rho_{\infty,m},\phi_{\infty,m})$ be the special solution to $\eqref{Geodesic-BB}$ defined as $\eqref{rho_infty,m}$ and $\eqref{phi_infty,m}$ on $[\delta,T]\subset(0,+\infty)$. Let $(\rho,\phi)$ be any solution to $\eqref{Geodesic-BB}$ such that $\mathrm{Ent}_{\gamma}(\rho(\delta))\geq \mathrm{Ent}_{\gamma}(\rho_{\infty,m}(\delta))$ and $\left.\frac{d}{dt}\right|_{t=\delta}\mathrm{Ent}_{\gamma}(\rho(t))\geq \left.\frac{d}{dt}\right|_{t=\delta}\mathrm{Ent}_{\gamma}(\rho_{\infty,m}(t))$. If $\mathrm{Ric}_{m,n}(L)\geq 0$, then we have the following comparison formula 
\begin{equation*}
\mathrm{Ent}_{\gamma}(\rho(t))\geq \mathrm{Ent}_{\gamma}(\rho_{\infty,m}(t)), \quad \forall t\in[\delta,T].  
\end{equation*}
\end{corollary}



\section{Hamiltonian and Lagrangian along the Langevin deformation}\label{Sec6}

In this section we prove the variational formulae of the Hamiltonian and the Lagrangian along the Langevin deformation of flows on the Wasserstein space over the Euclidean space or a compact Riemannian manifold. By Theorem \ref{existence}, for any $c\in (0, \infty)$,  the Langevin deformation of flows on $T\mathcal{P}_2(M, \mu)$ has a unique smooth solution (up to an additional constant) on $[0, T_c]\times \mathcal{P}_2^\infty(M, \mu)$. 

Consider the energy functional 
\begin{equation*}
	\mathcal{V}(\rho)=\int_M V(\rho) d\mu,
\end{equation*}
where $d\mu=e^{-f}dv$ is the weighted volume measure on $(M,g)$. Recall the following 
\begin{theorem}\label{Th-LL-T1}\cite{Li-Li} For $c>0$, define  the Hamiltonian and the Lagrangian as follows
\begin{eqnarray*}
H(\rho, \phi)&=&{c^2\over 2}\int_M |\nabla\phi(x)|^2\rho\hspace{0.2mm} d\mu+\mathcal{V}(\rho),\\
L(\rho, \phi)&=&{c^2\over 2}\int_M |\nabla\phi(x)|^2\rho\hspace{0.2mm} d\mu-\mathcal{V}(\rho).
\end{eqnarray*}
Let $(\rho(t), \phi(t))$ be a smooth solution to the Langevin deformation $(\ref{Langevin flow 0})$ on $\mathcal{P}_2(M, \mu)$. 
Then
\begin{eqnarray*}
{d\over dt}H(\rho(t), \phi(t))&=&-\int_M |\nabla\phi|^2\rho\hspace{0.2mm} d\mu,\\
{d\over dt} L(\rho(t), \phi(t))&=&-\int_M \left\langle \nabla\left(\phi+2 V^{\prime}(\rho)\right), \nabla \phi\right\rangle
\rho\hspace{0.2mm} d\mu,
\end{eqnarray*}
and 
\begin{eqnarray*}
{d^2\over dt^2}H(\rho(t), \phi(t))&=&{2\over c^2}
\int_M \left\langle \nabla\left(\phi+V^{\prime}(\rho)\right), \nabla \phi\right\rangle\rho\hspace{0.2mm} d\mu,\label{D2H}\\
{d^2\over dt^2}L(\rho(t), \phi(t))&=&2c^{2}\int_M \left|\nabla\left(\partial_t \phi+{1\over 2}|\nabla \phi|^2\right)\right|^2\rho\hspace{0.2mm} d\mu-2\nabla^2 \mathcal{V}(\rho)\left(\dot\rho, \dot\rho\right). 
\label{D2L}
\end{eqnarray*}
In particular, if $-\mathcal{V}$ is $K$-convex on $\mathcal{P}_2^\infty(M, \mu)$, i.e., the Hessian of $-\mathcal{V}$ on $\mathcal{P}_2^\infty(M, \mu)$ satisfies 
$$-\nabla^2 \mathcal{V}(\rho)\geq K,$$
 then
\begin{eqnarray*}
{d^2\over dt^2}L(\rho(t), \phi(t))\geq 2K\int_M |\nabla\phi|^2\rho\hspace{0.2mm}  d\mu+2c^2\int_M \left|\nabla\left(\partial_t \phi+{1\over 2}|\nabla\phi|^2\right)\right|^2\rho\hspace{0.2mm} d\mu.
\end{eqnarray*}
\end{theorem}

We may interpret $H$ and $L$ as the Hamiltonian and the Lagrangian 
of the Langevin deformation in an external potential 
$\mathcal{V}$ respectively. We may also interpret $L$ as the Hamiltonian of the Langevin deformation in 
an external 
potential $-\mathcal{V}$, and $H$ as the Lagrangian of the Langevin deformation in an external potential $-\mathcal{V}$.

Let $(\rho(t), \phi(t))$ be a smooth solution to the Langevin deformation $(\ref{Langevin flow 0})$ on $\mathcal{P}_2(M, d\mu)$. Recall the Hessian of $\mathcal{V}$ is given by $\eqref{Hess of V}$. Then by Theorem $\ref{Th-LL-T1}$, we have 
\begin{eqnarray*}
{d\over dt}H(\rho, \phi)&=&-\int_M |\nabla\phi|^2\rho d\mu,\\
{d^2\over dt^2}L(\rho, \phi)&=&2\int_M \left[c^{-2}|\nabla\phi+\nabla V^\prime(\rho)|^2\rho\right]d\mu\\
&&-2\int_M\left\{\left[\|\nabla\phi\|_{\mathrm{HS}}^2+\mathrm{Ric}(L)(\nabla\phi, \nabla\phi)\right]P(\rho)+(L\phi)^2P_2(\rho)\right\}d\mu.
\end{eqnarray*}

In particular, taking  $\mathcal{V}(\rho)=\mathrm{Ent}_\gamma(\rho)={1\over \gamma-1}
\int_{M}\rho^\gamma\hspace{0.2mm} d\mu$, we have
\begin{eqnarray*}
P(r)=(\gamma-1)U(r)=r^\gamma, \quad P_2(r)=(\gamma-1)^2U(r)=(\gamma-1)r^\gamma.
\end{eqnarray*}
Hence we can obtain the following  

\begin{theorem}\label{Th-LL-T2} Let $c>0$. Let $(\rho(t), \phi(t))$ be a smooth solution of the following equations
\begin{eqnarray*}
\partial_t \rho-\nabla_\mu^*(\rho\nabla \phi)&=&0,\label{zzz1}\\
c^2\left(\partial_t\phi+{1\over 2}|\nabla \phi|^2\right)&=&-\phi-{\gamma\over \gamma-1}\rho^{\gamma-1}.\label{zzz2}
\end{eqnarray*} 
Let
\begin{eqnarray*}
H(\rho(t), \phi(t))&=&{c^2\over 2}\int_M |\nabla\phi|^2\rho\hspace{0.2mm} d\mu+{1\over \gamma-1}
\int_{M}\rho^\gamma\hspace{0.2mm} d\mu,\\
L(\rho(t), \phi(t))&=&{c^2\over 2}\int_M |\nabla\phi|^2\rho\hspace{0.2mm} d\mu-{1\over \gamma-1}
\int_{M}\rho^\gamma\hspace{0.2mm} d\mu.
\end{eqnarray*}
Then
\begin{eqnarray*}
{d\over dt}H(\rho, \phi)&=&-\int_M |\nabla\phi|^2\rho d\mu,\\
{d^2\over dt^2}L(\rho, \phi)&=&2\int_M \left[c^{-2}|\nabla\phi+\gamma\rho^{\gamma-2}\nabla\rho|^2\rho\right]d\mu\\
&&-2\int_M\left[\|\nabla\phi\|_{\mathrm{HS}}^2+\mathrm{Ric}(L)(\nabla\phi, \nabla\phi)+(\gamma-1)(L\phi)^2\right] \rho^\gamma d\mu. 
\end{eqnarray*}
\end{theorem}

In the case $\mathcal{V}(\rho)=-\mathrm{Ent}_\gamma(\rho)=-{1\over \gamma-1}
\int_{M}\rho^\gamma\hspace{0.2mm} d\mu$,  we have Theorem \ref
{MT0}.

\section{The convergence results}
In this section, we state the convergence of the solution to the Langevin deformation $(\ref{Lang-gamma})$. We will first show the convergence in Euclidean space and then extend the results to compact manifolds. For simplicity, we consider the Laplace-Beltrami operator here instead of the Witten Laplacian. Under suitable assumptions on the potential function $f$, it is easy to prove that the convergence results also hold for the Witten Laplacian. Note that the $L^2$-adjoint of $\nabla$ with respect to the standard volume measure on $\mathbb{R}^n$ or a compact Riemannian manifold $(M, g)$ is given by $\nabla^*=-\nabla\cdot$. 
We get the Langevin deformation of flows as follows 
\beqna
\label{langevin.eqa1}\partial_t \rho + \nabla \cdot (\rho \nabla \phi) &=& 0, \\
\label{langevin.eqa2}c^{2}\left(\partial_t \phi + \frac{1}{2}|\nabla \phi|^{2}\right) &=& -\phi - \frac{\gamma}{\gamma-1}\rho^{\gamma-1},
\eeqna
on $T\mathcal{P}_2(\mathbb R^{n},dx)$ or $T\mathcal{P}_2(M,dv)$ with initial value
$$
\rho(0, x) = \rho_{0}(x)>0, \ \ \phi(0, x) = \phi_{0}(x).
$$

As pointed out in \cite{Li-Li}, the key point of the proof still relies on the close connection between the Langevin deformation of flows and the compressible Euler equation with damping. Thus we first prove the convergence for 
\beqna
\label{Euler with c1}\partial_t\rho+\nabla\cdot(\rho u)&=&0,\\
\label{Euler with c2}c^2\left(\partial_t u+u\cdot\nabla u\right)&=&-u-\frac{\gamma}{\gamma-1}\nabla\rho^{\gamma-1}, 
\eeqna
and then turn back to $(\ref{langevin.eqa1})$ and $(\ref{langevin.eqa2})$. 

Our aim is to prove that when $c \rightarrow 0$ and $c \rightarrow \infty$, the solution $(\rho, \phi)$ converges in a precise sense to those to the porous medium equation and to the geodesic flow respectively. 

\begin{theorem}\label{main.convergence}
Let $M$ be $\bbR^{n}$ or a compact Riemannian manifold. Let $s \in \bbN$ with $s > {n \over 2} +2$. Let $(\rho^{c}, \phi^{c})$ be the local solution to the initial value problem of Langevin deformation of flows $(\ref{langevin.eqa1})$ and $(\ref{langevin.eqa2})$. 
\bitem
{\item
Given the initial value $(\rho_{0}, \phi_{0}) \in TP^{\infty}_{2}(M, \mu)$ satisfying $\rho_{0} \in H^{s}(M)$, $\phi_{0} \in H^{s+1}(M)$, there exists $T > 0$ which is independent of $c$, such that as $c \rightarrow 0$,  
\begin{equation}\label{convergence as c to 0}
	\sup_{t \in [0, T]}  \|\rho^{c} - \rho^{0}\|_{L^{1}}  \rightarrow 0,
\end{equation}
where $\rho^{0} \in C( [0, T], C^{3}(M))$ is the solution to the porous medium equation with the initial value $\rho_{0}$. 
}
{\item
Given initial value $(\rho_{0}, \phi_{0}) \in T\mathcal{P}^{\infty}_{2}(M, \mu)$ satisfying
$$
\left\|p_0 \right\|_{H^s}^2+\left\|u_0\right\|_{H^s}^2 \leq M
$$
for some constant $M$, where
$$p_0=\frac{\rho_0^{\gamma-1}-1}{\gamma-1},\quad u_0=\nabla\phi_0, $$ 
there exists some $T>0$, which is independent of $c>1$, such that when $c \rightarrow \infty$, ~$\rho^c d x$ weakly converges to $\rho^{\infty} d x$ in $\mathcal{C}\left([0, T], \mathcal{P}\left(\mathbb{R}^d\right)\right)$, and $u^c$ converges to $u^{\infty}(x, t) \in \mathcal{C}\left([0, T], H^s\right) \cap$ $\mathcal{C}^1\left([0, T], H^{s-1}\right)$ in $\mathcal{C}\left([0, T], H^{s^{\prime}-1}\left(B_R\right)\right)$ for any $R>0, s^{\prime}<s$. Moreover, $\left(\rho^{\infty}, u^{\infty}\right)$ satisfies
\begin{equation}
\begin{aligned}
\partial_t \rho^{\infty}+\nabla \cdot\left(\rho^{\infty} u^{\infty}\right) & =0 \\
\partial_t u^{\infty}+u^{\infty} \cdot \nabla u^{\infty} & =0 \\
\rho^{\infty}(0, x)=\rho_0(x), \quad u^{\infty}(0, x) & =u_0(x) .
\end{aligned}
\end{equation}
}
\eitem
\end{theorem}


\subsection{Convergence as $c$ approaches $0$}

Inspired by the works of Lattanzio-Tzavaras \cite{L-T2013}, who applied the method of relative entropy to study the diffusive limit of entropy weak solutions to compressible isentropic Euler equation with damping on $\mathbb T^3$ or $\mathbb R$. They consider these two case to avoid further technicalities. Actually, if the solution $(\rho_c(t),u_c(t))$ to $\eqref{Lang-gamma}$ is in $ H^s(M)\times H^s(M)$, then we can verify the calculation in \cite{L-T2013} by applying Sobolev inequality and extend their results to the solution of $\eqref{Lang-gamma}$ when $M$ is a compact Riemannian manifold or $\mathbb R^n$ for $t\in[0,T]$ where $T$ is small enough. To compare the model in \cite{L-T2013} with Langevin deformation $\eqref{Lang-gamma}$, we change the parameter $\varepsilon$ in \cite{L-T2013} to $c$. Rewriting the porous medium equation by 
\begin{equation}\label{porous medium with Darcy law}
\begin{cases}
\partial_t \bar\rho+\frac{1}{c}\nabla\cdot (\bar\rho\bar u)=0,\\
\bar u=-c\nabla P(\bar\rho). 	
\end{cases}
\end{equation}
Lattanzio and Tzavaras proved the following 

\begin{theorem}\label{L-T2013}\cite{L-T2013}
Assume the initial value $(\rho_0, u_0)$ satisfies
\begin{equation*}
	\int_{\mathbb T^3}\rho_0dx<+\infty,
\end{equation*}
\begin{equation*}
	\int_{\mathbb T^3} \left(\frac{1}{2}\rho_0\|u_0\|^2+\frac{1}{\gamma-1}\rho_0^{\gamma} \right) dx<+\infty.
\end{equation*}
Let $(\rho(t),u(t)), t\in[0,T]$ be a weak entropy solution to 
\begin{equation*}
\begin{cases}
\partial_t \rho+\frac{1}{c}\nabla\cdot(\rho u)=0,\\
\partial_t(\rho u)+\frac{1}{c}\nabla\cdot(\rho u\otimes u^c)=-\frac{1}{c}\nabla P(\rho)-\frac{1}{c^2}\rho u,
\end{cases}
\end{equation*}
with $P(\rho)=\rho^\gamma$. Let $(\bar\rho,\bar u)$ be the smooth solution of $\eqref{porous medium with Darcy law}$ with initial value $(\bar\rho_0,\bar u_0)$ that avoids vacuum. Then for $t\in [0,T]$, it holds 
\begin{equation*}
\int_{\mathbb T^3} \left(h(\rho|\bar\rho)+\frac{1}{2}\rho\|u-\bar u\|^2\right) dx\leq C\left(\int_{\mathbb T^3} \left(h(\rho_0|\bar\rho_0)+\frac{1}{2}\rho_0\|u_0-\bar u_0\|^2 \right) dx+c^4\right),
\end{equation*}
where $C$ is a constant which does not depend on $c$ and $h(\rho|\bar\rho)$ is the relative entropy defined by
\begin{equation*}
h(\rho|\bar\rho)=h(\rho)-h(\bar\rho)-h^\prime(\bar\rho)(\rho-\bar\rho),
\end{equation*} 
where 
\begin{equation*}
	h(\rho)=\frac{1}{\gamma-1}\rho^\gamma,\quad \gamma>1.
\end{equation*}
\end{theorem}

Now we prove the convergence as $c\to0$. 

\begin{proof}[Proof of \eqref{convergence as c to 0}]
If we take $(\rho^c(t,x),u^c(t,x))=\left(\rho(t,x),\frac{1}{c}u(t,x)\right)$, then one can easily check that $(\rho^c,u^c)$ is a solution to the Langevin deformation $\eqref{Lang-gamma}$. We may also take $(\bar\rho^c(t,x),\bar u^c(t,x))=\left(\bar \rho(t,x),\frac{1}{c}\bar u(t,x)\right)$, then $(\bar\rho^c(t,x),\bar u^c(t,x))$ is a solution to the porous medium equation
\begin{equation*}
\begin{cases}
\partial_t \bar\rho^c+\nabla\cdot (\bar\rho^c\bar u^c)=0,\\
\bar u^c=-\nabla P(\bar\rho^c). 	
\end{cases}
\end{equation*}
Thus by Theorem $\ref{L-T2013}$, we can derive  
\begin{equation*}
\int_{\mathbb R^n}\left( h(\rho^c|\bar\rho^c)+\frac{c^2}{2}\rho^c\|u^c-\bar u^c\|^2 \right)dx\leq C\left(\int_{\mathbb T^3}\left( h(\rho_0^c|\bar\rho_0^c)+\frac{c^2}{2}\rho_0^c\|u_0^c-\bar u_0^c\|^2 \right)dx+c^4\right).
\end{equation*}
By choosing $(\rho^c(0,x),u^c(t,x))=(\bar \rho^c(0,x),\bar u^c(0,x))$, we obtain
\begin{equation*}
\int_{\mathbb R^n}\left( h(\rho^c|\bar\rho^c)+\frac{c^2}{2}\rho^c\|u^c-\bar u^c\|^2\right) dx\leq Cc^4.	
\end{equation*}
Thus by Csiszár-Kullback-Pinsker's inequality, we can obtain the convergence result \eqref{convergence as c to 0}.
\end{proof}

\begin{remark}
We point out that under suitable scaling, the equations $(\ref{langevin.eqa1})$ and $(\ref{langevin.eqa2})$ are exactly the isentropic Euler equations considered in Lin and Coulombel \cite{C-G2013}.
Then following the method in \cite{C-G2013,Li-Li}, when $c \rightarrow 0$, we are able to show the convergence of the global solution to the porous medium equation with initial value. Since the requirement of small initial value is not natural in our case, we omit the details here to save length of the paper. 

\end{remark}

\subsection{Convergence as $c$ approaches $+\infty$}

Next, we use the same strategy in \cite{Li-Li} and follow the proof in Klainerman-Majda \cite{K-M1981,K-M1982,Majda1984} on the low Mach number limit of the Euler equations to complete the proof of convergence of the Langevin deformation of flows with R\'enyi entropy when $c$ approaches $\infty$. Notice that if we let  
\begin{equation}
	p(\rho)=\frac{\rho^{\gamma-1}-1}{\gamma-1},
\end{equation}
then $(\ref{Euler with c1})$ and $(\ref{Euler with c2})$ becomes 
\begin{equation}\label{Euler with gamma}
\left\{\begin{array}{l}
\frac{1}{(\gamma-1)p+1}\left(\partial_{t} p+u \cdot \nabla p\right)+ \nabla \cdot u=0, \\
\frac{c^2}{\gamma}\left(\partial_{t} u+u \cdot \nabla u\right)+\nabla p+\frac{u}{\gamma}=0,
\end{array}\right.
\end{equation}
which is well-defined since $\rho>0$ and $\gamma>1$. Denote $U=(p,u)^{T}$, $U:\mathbb R^n\times[0,T]\to G\subset\mathbb R^{n+1}$, where $G$ is called the state space. Rewrite \eqref{Euler with gamma} into the following symmetric hyperbolic system:
\begin{equation}\label{sym hyper}
\left\{\begin{array}{l}
	A_0(U,c)\partial_tU+\sum_{j=1}^{n}A_j(U,c)\partial_jU+B(U,c)U=0,\\
	U(0,x)=U_0(x)\in G_1, \bar{G_1}\subset\subset G,
\end{array}\right.
\end{equation}
where
\begin{equation*}
A_0(U,c)=
\left(\begin{array}{cc}
\frac{1}{(\gamma -1)p+1} & 0\\
0 & \frac{c^2}{\gamma} \mathrm I_{n\times n}
\end{array}\right), \quad
B(U,c)=B=
\left(\begin{array}{cc}
0 & 0\\
0 & \frac{1}{\gamma}\mathrm I_{n\times n}
\end{array}\right),
\end{equation*}

\begin{equation*}
A_j(U,c)=
\left(\begin{array}{cccccc}
\frac{u^j}{(\gamma-1)p+1} & 0 & \cdots & 1 & \cdots & 0 \\
0 & \frac{c^2}{\gamma} u^j & \cdots & 0 & \cdots & 0 \\
\vdots & \vdots & \ddots & \vdots & & \vdots \\
1 & 0 & \cdots & \frac{c^2}{\gamma} u^j & \cdots & 0 \\
\vdots & \vdots & & \vdots & \ddots & \vdots \\
0 & 0 & \cdots & 0 & \cdots & \frac{c^2}{\gamma} u^j
\end{array}\right)_{(n+1) \times(n+1)}.
\end{equation*}
Recall that the existence and uniqueness of solution to the Cauchy problem for symmetric hyperbolic systems is originated to Kato~\cite{Kato1975} who dealt with abstract evolution equations. For our case, we refer to the following theorem by Majda \cite{Majda1984}. 
\begin{theorem}\cite{Majda1984} 
Assume $U_0 \in H^s, s>\frac{n}{2}+1$ and $U_0(x) \in G_1, \bar{G}_1 \subset \subset G$. $A_0, A_j$ are symmetric for $j\in\{1,\cdots,n\}$ and there exists some constant $C$ such that $C^{-1}\mathrm I\leq A_0(U)\leq C\mathrm I$ uniformly for $U\in G_1$. Then there is a time interval $[0, T]$ with $T>0$, so that the equations in $(\ref{sym hyper})$ have a	unique classical solution $U(t,x)\in C^1([0,T]\times\mathbb R^n, G)$ with $U(t,x)\in G_2$, $\bar{G_2}\subset\subset G$ for $(t,x)\in[0,T]\times\mathbb R^n$. Furthermore,
$$U \in C\left([0, T], H^s\right) \cap C^1\left([0, T], H^{s-1}\right)$$
and $T$ depends on $\left\|U_0\right\|_s$ and $G_1$, i.e., $T=T\left(\left\|U_0\right\|_s, G_1\right)$.
\end{theorem}
For brevity, we denote $H^s(\mathbb R^n)$ by $H^s$, and we use the following notations: 
\begin{equation*}
	\|u(t)\|_0:=\|u(t,\cdot)\|_{L^2}, \quad \|u\|_0:=\sup_{t\in [0,T]}\|u(t)\|_0.
\end{equation*}
\begin{equation*}
	\|u(t)\|_s^2:=\|u(t,\cdot)\|_{H^s}^2=\sum_{|\alpha|=0}^{s}\int_{\mathbb R^n} |D^\alpha u(t,x)|^2 dx, \quad \|u\|_s:=\sup_{t\in [0,T]}\|u(t)\|_s.
\end{equation*}


We will denote $(\rho,u)=(\rho^c,u^c)$ for convenience. According to Majda \cite{Majda1984}, the key step to prove the convergence result is the following uniform a priori estimate.
\begin{proposition}\label{Prop7.9}
Assume that for initial value $\left(\rho_0, u_0\right)$, there exist two constants $M>0$ and $s>\frac{n}{2}+1$, such that
\begin{equation}\label{initial data}
	c^{-2}\left\|p_0 \right\|_{H^s}^2+\left\|u_0\right\|_{H^s}^2 \leq M .
\end{equation}
Then there exists $T>0$, which is independent of $c$, such that for any $c \geq 1$, the equation $(\ref{sym hyper})$ admits classical solution $U=(p, u) \in \mathcal{C}\left([0, T], H^s\right) \cap \mathcal{C}^1\left([0, T], H^{s-1}\right)$, satisfying
\begin{equation}\label{thm result}
	\sup _{t\in [0, T]}\left[ c^{-2}\left\|p \right\|_{H^s}^2+\left\|u\right\|_{H^s}^2 \right]\leq R,
\end{equation}
where $R$ is a constant independent of $c$.
\end{proposition}
To see this, we first prove the following a priori estimate as in \cite{Majda1984}.
\begin{lemma}
Assume that the conditions in Theorem $\ref{main.convergence}$ hold, and the equation $(\ref{sym hyper})$ admits local classical solution $\left(p, u\right)$ on $\left[0, T_c^*\right]$, satisfying
\begin{equation}\label{assumption}
	\sup_{t\in\left[0, T_c^*\right]} \left[c^{-2}\left\|p \right\|_{H^s}^2+\left\|u\right\|_{H^s}^2 \right] \leq 2 M,
\end{equation}
for all $c \geq 1$. Then there exists a constant $C>0$, which is independent of $c$, such that
\begin{equation}\label{lemma result}
	\sup _{t\in\left[0, T_c^*\right]} \left[c^{-2}\left\|p \right\|_{H^s}^2+\left\|u\right\|_{H^s}^2 \right] \leq\left(c^{-2}\left\|p_0 \right\|_{H^s}^2+\left\|u_0\right\|_{H^s}^2\right) e^{CM T_c^*}.
\end{equation}
\end{lemma}

\begin{proof}
The proof relies on the classical energy estimate method. First, multiplying $U$ on both sides of the equation in $(\ref{sym hyper})$, it holds
\begin{equation*}
	\left\langle A_0\partial_tU,U\right\rangle=-\sum_{j=1}^{n}\left\langle A_j\partial_jU,U\right\rangle-\left\langle B U,U\right\rangle,
\end{equation*}
where $\langle\cdot,\cdot\rangle$ means the $L^2$ inner product. By the symmetry of $A_j$ for $j=0,1,\dots,n$ and the integration by parts formula, we can prove 
\begin{equation}\label{energy method}
\begin{aligned}
\partial_t\left\langle A_0 U, U\right\rangle 
=\left\langle\mathrm{div}\bar{A} U, U\right\rangle-2\left\langle B U, U\right\rangle,
\end{aligned}
\end{equation}
where $\bar{A}=(A_0,A_1,\cdots,A_n)$ and $\mathrm{div}\bar{A}=\partial_tA_0+\sum_{j=1}^{n}\partial_jA_j$. 
By direct computation,
\begin{equation*}
\begin{aligned}
\partial_tA_0+\sum_{j=1}^{n}\partial_jA_j=
\left(\begin{array}{cc}
\frac{\gamma \nabla\cdot u}{(\gamma-1)p+1} & 0\\
0 & \frac{c^2}{\gamma}\nabla\cdot u\mathrm{I}_{n\times n}
\end{array}\right).
\end{aligned}
\end{equation*}
Thus 
\begin{equation*}
	\left\langle\mathrm{div}\bar{A} U, U\right\rangle
	=\int_{\mathbb R^n} \left[\frac{\gamma \nabla\cdot u}{(\gamma-1)p+1}|p|^2+\frac{\nabla\cdot u}{\gamma}c^2|u|^2\right]dx.
\end{equation*}
By the Sobolev inequality, for $s>\frac{n}{2}+1$, it holds
\begin{equation*}
	\|\nabla\cdot u\|_{L^\infty}\leq C\|\nabla u\|_{L^\infty}\leq C_s\|u\|_{H^s}\leq 2C_sM
\end{equation*}
uniformly for $t\in[0,T_c^*]$. Since $(\gamma-1)p+1=\rho^{\gamma-1}\geq \varepsilon>0$, it holds
\begin{equation}\label{div A bdd}
	\left\langle\mathrm{div}\bar{A} U, U\right\rangle\leq C_1M\left(\|p\|_0^2+c^2\|u\|_0^2\right).
\end{equation}
Notice that $\left\langle A_0 U, U\right\rangle$ is equivalent to $\|p\|_0^2+c^2\|u\|_0^2$. In fact, $\|p(t)\|_{L^\infty}\leq C_s\|p(t)\|_s$, and $\sup_{t\in[0,T_c^*]}\|p(t)\|_{L^\infty}\leq \sup_{t\in[0,T_c^*]}\|p(t)\|_s\leq 2M$. Hence
\begin{equation}\label{equivalent norm}
	\lambda \left(\|p\|_0^2+c^2\|u\|_0^2\right)\leq \left\langle A_0 U, U\right\rangle\leq\Lambda \left(\|p\|_0^2+c^2\|u\|_0^2\right)
\end{equation}
for some constant $\lambda$ and $\Lambda$ which are independent of $c$. Noticing $B$ is positive definite, so we have 
\begin{equation*}\label{zero oder}
\partial_t\left\langle A_0 U, U\right\rangle\leq C_1M\left(\|p\|_0^2+c^2\|u\|_0^2\right)\leq C_2M\left\langle A_0 U, U\right\rangle.
\end{equation*}
Applying Gronwall's inequality, we obtain
\begin{equation*}
	\sup_{t\in[0,T_c^*]}\left\langle A_0 U, U\right\rangle\leq e^{C_2M T_c^*}\sup_{t\in[0,T_c^*]}\left\langle A_0 U_0, U_0\right\rangle.
\end{equation*}
Equivalently,  we conclude that
\begin{equation*}
	\|p\|_0^2+c^2\|u\|_0^2\leq e^{C_3MT_c^*}\left(\|p_0\|_0^2+c^2\|u_0\|_0^2\right).
\end{equation*}
Note that $C_1, C_2, C_3$ are all independent of $c$, and they all depend on $\varepsilon, \gamma, C_s$, where $C_s$ is the Sobolev constant. \\

Then we estimate $\|D^\alpha U\|_{0}$ by the same energy method. Since $A_0$ is positive definite, we rewrite $(\ref{sym hyper})$ by 
\begin{equation*}
	\partial_t U+\sum_j\left(A_0^{-1} A_j \partial_j U\right)+A_0^{-1} B U=0.
\end{equation*}
For $1\leq \alpha\leq s$, applying differential operator $D^\alpha$ on both sides, we obtain
\begin{equation*}\label{Du eq}
	\partial_t D^\alpha U+\sum_j D^\alpha\left(A_0^{-1} A_j \partial_j U\right)+D^\alpha\left(A_0^{-1} B U\right)=0.
\end{equation*}
Multiply $A_0$ on both sides, and we can obtain 
\begin{equation*}
	A_0 \partial_t D^\alpha U+\sum_j A_j \partial_j D^\alpha U=F_\alpha,
\end{equation*}
where
\begin{equation*}
	F_\alpha=-A_0\left[\sum_j D^\alpha\left(A_0^{-1} A_j \partial_j U\right)-A_0^{-1} A_j \partial_j D^\alpha U\right]-B D^\alpha U.
\end{equation*}
By the same skill in $(\ref{energy method})$, we have
\begin{equation}\label{Du eq}
	\partial_t\left\langle A_0 D^\alpha U, D^\alpha U\right\rangle=\left\langle\operatorname{div} \bar{A} D^\alpha U, D^\alpha U\right\rangle+2\left\langle F_\alpha, D^\alpha U\right\rangle,
\end{equation}
where
\begin{equation*}
	\left\langle F_\alpha, D^\alpha U\right\rangle=-\left\langle\sum_j D^\alpha\left(A_0^{-1} A_j \partial_j U\right)-A_0^{-1} A_j \partial_j D^\alpha U, A_0 D^\alpha U\right\rangle-\left\langle BD^\alpha U,D^\alpha U\right\rangle. 
\end{equation*}


It suffices to estimate the first term since $B$ is positive definite.
By direct computation, we obtain
\begin{equation*}
A_0^{-1}A_j=
\left(\begin{array}{cccccc}
u^j & 0 & \cdots & (\gamma-1)p+1 & \cdots & 0 \\
0 & u^j & \cdots & 0 & \cdots & 0 \\
\vdots & \vdots & \ddots & \vdots & & \vdots \\
\frac{\gamma}{c^2} & 0 & \cdots & u^j & \cdots & 0 \\
\vdots & \vdots & & \vdots & \ddots & \vdots \\
0 & 0 & \cdots & 0 & \cdots & u^j
\end{array}\right)\\
=:
\left(\begin{array}{cc}
u^j & M^j_{12}(p)\\
c^{-2}\gamma e_j & M^j_{22}(u)
\end{array}\right)
\end{equation*}
where $e_j=(0,\cdots,1,\cdots,0)^T$ and $\{M^j_{kl}, k,l=1,2; j=1,\cdots,n\}$ are all bounded and smoothly depend on $(\rho,u)$. 
Then we can obtain 
\begin{equation*}
\begin{aligned}
\left\langle F_\alpha, D^\alpha U\right\rangle
\leq&\sum_{j=1}^n\left|\frac{1}{(\gamma-1)p+1}\left\langle D^\alpha\left(u^j\partial_j p\right)-u^jD^\alpha \partial_jp, D^\alpha p\right\rangle\right.\\
&+\frac{1}{(\gamma-1)p+1}\left\langle D^\alpha\left(M^j_{12}(p)\partial_j u\right)-M^j_{12}(p)D^\alpha \partial_ju, D^\alpha p\right\rangle\\
&\left.+\frac{c^{2}}{\gamma} \left\langle D^\alpha\left(M^j_{22}(u)\partial_j u\right)-M^j_{22}(u)D^\alpha \partial_ju, D^\alpha u\right\rangle\right|.
\end{aligned}	
\end{equation*}
 Since $(\gamma-1)p+1\geq \varepsilon>0$, we obtain
\begin{equation*}
\begin{aligned}
\left\langle F_\alpha, D^\alpha U\right\rangle
\leq&C_4\sum_{j=1}^n\Big[\|D^\alpha\left(u^j\partial_j p\right)-u^jD^\alpha \partial_jp\|_0 \|D^\alpha p\|_0\\
&+\|D^\alpha\left(M^j_{12}(p)\partial_j u\right)-M^j_{12}(p)D^\alpha \partial_ju\|_0 \|D^\alpha p\|_0\\
&+c^{2}\|D^\alpha\left(M^j_{22}(u)\partial_j u\right)-M^j_{22}(u)D^\alpha \partial_ju\|_0 \|D^\alpha u\|_0\Big].
\end{aligned}		
\end{equation*}
Then by the Moser-type Calculus inequalities (see e.g. Proposition 2.1 of \cite{Majda1984}), we have
\begin{equation*}
\begin{aligned}
\left\langle F_\alpha, D^\alpha U\right\rangle
\leq&C_5\sum_{j=1}^n\Big[\left(|Du^j|_{L^\infty}\|D^{s-1}\partial_jp\|_0+|\partial_j p|_{L^\infty}\|D^su^j\|_0\right)\|D^\alpha p\|_0\\
&+\left(|DM^j_{12}|_{L^\infty}\|D^{s-1}\partial_ju\|_0+|\partial_j u|_{L^\infty}\|D^sM^j_{12}\|_0\right) \|D^\alpha p\|_0\\
&+c^{2}\left(|DM^j_{22}|_{L^\infty}\|D^{s-1}\partial_ju\|_0+|\partial_j u|_{L^\infty}\|D^sM^j_{22}\|_0\right) \|D^\alpha u\|_0\Big].
\end{aligned}		
\end{equation*}
Substituting $M^j_{12}, M^j_{22}$ and applying Sobolev inequality, we derive
\begin{equation*}
\begin{aligned}
\left\langle F_\alpha, D^\alpha U\right\rangle
\leq&C_6\Big[\left(\|u\|_s\|D^s p\|_0+\|p\|_{s}\|D^s u\|_0\right)\|D^\alpha p\|_0\\
&+\left(\|p\|_s\|D^s u\|_0+\|u\|_{s}\|D^s p\|_0\right) \|D^\alpha p\|_0\\
&+c^{2}\left(\|u\|_s\|D^s u\|_0+\|u\|_{s}\|D^s u\|_0\right) \|D^\alpha u\|_0\Big]\\
\leq &C_7M\Big[\|p\|_s \|D^\alpha p\|_0+c^{2}\|u\|_s \|D^\alpha u\|_0\Big],
\end{aligned}	
\end{equation*}
where the last line is due to the assumption $(\ref{assumption})$.
Sum over $0\leq|\alpha|\leq s$ and combine with $(\ref{div A bdd})$,  $(\ref{equivalent norm})$ and $(\ref{Du eq})$, then we obtain
\begin{equation*}
\begin{aligned}
	\partial_t\left(\frac{1}{(\gamma-1)p+1}\|p\|_s^2+\frac{c^2}{\gamma}\|u\|_s^2\right)
	&\leq C_1M\left(\|p\|_s^2+\|u\|_s^2\right)+C_7 M\left(\|p\|_s^2+c^2\|u\|_s^2\right)\\
	&\leq C_{8}M\left(\frac{1}{(\gamma-1)p+1}\|p\|_s^2+\frac{c^2}{\gamma}\|u\|_s^2\right).
\end{aligned}
\end{equation*}
Then by Gronwall's inequality, we conclude
\begin{equation*}
	\|p\|_s^2+c^2\|u\|_s^2\leq e^{C_{9}Mt}\left(\|p_0\|_s^2+c^2\|u_0\|_s^2\right),
\end{equation*}
which implies $(\ref{lemma result})$. Note that $C_1, C_2, C_3$ are all independent of $c$, and $C_i=C_i(\varepsilon,\gamma)$ for $i=4,5,6$ and $C_j=C_j(\varepsilon,\gamma,n,C_s)$ for $j=7,\dots,9$, where $C_s$ is the Sobolev constant. 
\end{proof}
Now we are ready to prove Proposition $\ref{Prop7.9}$.
\begin{proof}
We choose $T_c^*$ in \eqref{assumption} such that
\begin{equation*}
		e^{C(M)T_c^*}=2.
\end{equation*}	
Since $C$ is independent of $c$, $T_c^*=T$ is also independent of $c$. Then under the initial condition $(\ref{initial data})$,  the estimate $(\ref{lemma result})$ automatically leads to the additional a priori estimate in $(\ref{assumption})$. Finally $(\ref{thm result})$ is obtained by taking $R=2M$. 
\end{proof}

Now we can prove the rest part (i.e., as $c\to+\infty$) of Theorem $\ref{main.convergence}$ . Since the rest of the proof follows the same arguments as in \cite{Li-Li}, we end our proof here.

\section*{Acknowledgement: } The second and the third named authors would like to thank Professors D. Bakry, J.-M. Bismut, G.-Q. Chen and F.-M. Huang for valuable discussions. This work has been done since the first named author did her PhD  thesis at the Academy of Mathematics and Systems Science (AMSS), Chinese Academy of Sciences (CAS). She is grateful to AMSS and CAS for their support and to Professors K. Kuwae and J. Masamune for providing her postdoctoral fellowships during which this paper is completed.


\end{document}